\documentclass[final,leqno]{siamltex}

\usepackage[leqno]{amsmath}
\usepackage{graphicx}
\usepackage{graphics}
\usepackage{epstopdf}
\usepackage{subfigure}
\usepackage{threeparttable}

\usepackage{bm}
\usepackage{amssymb}
\usepackage{leftidx}
\usepackage{empheq}
\usepackage{color}

\numberwithin{equation}{section}

\title{The fast scalar auxiliary variable approach with unconditional energy stability
for nonlocal Cahn-Hilliard equation
        \thanks{
We would like to acknowledge the assistance of volunteers in putting together this example manuscript and supplement. This work was supported in part by the National Natural Science Foundation of China under Grants 91630207, 11471194 and 11571115, by the National Science Foundation under Grant DMS-1216923, by the OSD/ARO MURI Grant W911NF-15-1-0562, and by Taishan Scholars Program of Shandong Province of China. The authors Z. Liu and X. Li thank for the financial support from China Scholarship Council.}}

      \author{Zhengguang Liu
             \thanks{School of Mathematics and Statistics, Shandong Normal University, Jinan, China. Second address: School of Mathematics, Shandong University, Jinan 250100, China Email: liuzhgsdu@yahoo.com}.
             \and
             Aijie Cheng
             \thanks{Corresponding Author. School of Mathematics, Shandong University, Jinan 250100, China. Email: aijie@sdu.edu.cn}.
             \and
             Xiaoli Li
             \thanks{Fujian Provincial Key Laboratory on Mathematical Modeling and High Performance Scientific Computing and School of Mathematical Sciences,Xiamen University, Xiamen, Fujian, 361005, China. Email: xiaolisdu@163.com}. }

\begin{document}

\maketitle

\begin{abstract}
Comparing with the classical local gradient flow and phase field models, the nonlocal models such as nonlocal Cahn-Hilliard equations equipped with nonlocal diffusion operator can describe more practical phenomena for modeling phase transitions. In this paper, we construct an accurate and efficient scalar auxiliary variable approach for the nonlocal Cahn-Hilliard equation with general nonlinear potential. The first contribution is that we have proved the unconditional energy stability for nonlocal Cahn-Hilliard model and its semi-discrete schemes carefully and rigorously. Secondly, what we need to focus on is that the non-locality of the nonlocal diffusion term will lead the stiffness matrix to be almost full matrix which generates huge computational work and memory requirement. For spatial discretizaion by finite difference method, we find that the discretizaition for nonlocal operator will lead to a block-Toeplitz-Toeplitz-block (BTTB) matrix by applying four transformation operators. Based on this special structure, we present a fast procedure to reduce the computational work and memory requirement. Finally, several numerical simulations are demonstrated to verify the accuracy and efficiency of our proposed schemes.
\end{abstract}



\pagestyle{myheadings}
\thispagestyle{plain}
\markboth{ZHENGGUANG LIU, AIJIE CHENG AND XIAOLI LI} {FAST SCALAR AUXILIARY VARIABLE APPROACH}
  \section{Introduction}
Phase field models have been extensively applied to study the dynamics of different material phases via order parameters, such as phase transformations in binary alloys  \cite{hu2001phase,jou1997microstructural}; epitaxial thin film growth \cite{torabi2009new,wise2005quantum}; crystal faceting \cite{bollada2018faceted}; multi-phase fluid flow \cite{badalassi2003computation,chen2016efficient}. They can describe the evolution of complex, morphology-changing surfaces and construct a general framework to take the consideration of more physical effects.

There exist many advantages in the phase field models from the mathematical point of view. Specially, since the phase field models are usually energy stable (thermodynamics-consistent) and well-posed, which is based on the energy variational approach, it is possible to perform effective numerical analysis and carry out reliable and accurate computer simulations. The significant goal is to preserve the energy stable property at the discrete level irrespectively of the coarseness of the discretization in time and space. Schemes with this property is extremely preferred for solving diffusive systems due to the fact that it is not only critical for the numerical scheme to capture the correct long time dynamics of the system, but also supplies sufficient flexibility for dealing with the stiffness issue. Moreover, the noncompliance of energy dissipation laws may lead to spurious numerical approximations if the mesh or time step sizes are not controlled carefully. However, due to the thin interface, it is a quite difficult issue to construct unconditionally energy stable schemes for phase field models such as the Cahn-Hilliard and Allen-Cahn equations with general nonlinear potential. Many efforts had been done in order to solve this problem (c.f. \cite{chen2015decoupled,he2007large,liu2018time,shen2012second,shen2018scalar,shen2010numerical,zhao2016numerical}).

Recently, nonlocal phase filed models such as nonlocal Cahn-Hilliard equation have attracted more and more attentions and been used in many fields involving physics, materials science, finance and image processing \cite{chen2018accurate,chen2018power,du2018stabilized}. Many important phenomena are well described by nonlocal model with nonlocal diffusion term. Bates and Han \cite{bates2005dirichlet,bates2005neumann} investigated the well-posedness of the nonlocal Cahn-Hilliard equations equipped with Dirichlet or Neumann boundary condition. Stabilized linear semi-implicit schemes for the nonlocal Cahn-Hilliard equation were considered by Du et al. and the energy stabilities were established for two methods in the fully discrete sense \cite{du2018stabilized}. Guan et al. \cite{guan2014second} presented second-order accurate, unconditionally uniquely solvable and unconditionally energy stable schemes for the nonlocal Cahn-Hilliard and nonlocal Allen-Cahn equations for a large class of interaction kernels. Some other relative models such as nonlocal Cahn-Hilliard-Navier-Stokes \cite{frigeri2016nonlocal,frigeri2012global,frigeri2013strong} have also been considered and analyzed by many researchers.

The classical Cahn-Hilliard equation is a nonlinear, fourth order in space, parabolic partial differential equation which is often used as a diffuse interface model for the phase separation of a binary alloy:
\begin{equation*}
\frac{\partial \phi}{\partial t}+M(-\Delta)(-\epsilon^2\Delta \phi+F'(\phi))=0, \quad(\textbf{x},t)\in\Omega\times Q,
\end{equation*}
where $Q=(0,T]$, $M$ is the mobility constant, the interface width satisfies $O(\epsilon)$, which is small compared to the characteristic length of the laboratory scale. The phase-field $\phi$ represents the difference of local relative concentrations for the two components of the mixture such that $\phi=\pm1$ correspond to the pure phases of the material while $\phi \in (-1,1)$ corresponds to the transition between the two phases. $\textbf{x} \in \Omega \subseteq \mathbb{R}^d$. $F(\phi)$ is the nonlinear bulk potential and the most commonly used form Ginzburg-Landau double-well type potential is defined as follows  \cite{liu2003phase,minjeaud2013unconditionally,yang2017numerical}:
\begin{equation*}
F(y)=\frac{1}{4}(y^2-1)^2,~y\in(-\infty, \infty),
\end{equation*}
The free energy takes the form:
\begin{equation*}
E(\phi)=\int_{\Omega}(\frac{1}{2}\epsilon^2|\nabla \phi|^2+F(\phi))d\textbf{x}.
\end{equation*}

By replacing the Laplacian in the above local energy by the nonlocal diffusion operator $\mathcal{L}$ in \cite{du2012analysis}, one can obtain the nonlocal free energy functional as follows \cite{bates2009numerical,bates2006nonlocal,du2018stabilized}:
\begin{equation}\label{section1_e_energy functional}
E(\phi)=\int_{\Omega}\left(F(\phi)+\frac{\epsilon^2}{4}\int_\Omega J(\textbf{x}-\textbf{y})\left[\phi(\textbf{x})-\phi(\textbf{y})\right]^2d\textbf{y}\right)d\textbf{x},
\end{equation}
where the kernel $J$ satisfies the following conditions:

(\romannumeral1) $J(\textbf{x})\geq0$, $\forall \textbf{x}\in\Omega$;

(\romannumeral2) $J(\textbf{x}-\textbf{y})=J(\textbf{y}-\textbf{x})$;

(\romannumeral3) $J$ is $\Omega$-periodic.

In addition, we can also replace the first Laplician in classical Cahn-Hilliard model by another nonlocal operator $\mathcal{L}$ to obtain a general nonlocal Cahn-Hilliard model:
\begin{equation*}
\frac{\partial \phi}{\partial t}+M\mathcal{L}(-\epsilon^2\mathcal{L}\phi+F'(\phi))=0, \quad(\textbf{x},t)\in\Omega\times J,
\end{equation*}
where the nonlocal operator $\mathcal{L}$ on the function $u(\textbf{x})$ has been introduced by many articles such as \cite{du2012analysis,du2013nonlocal}
\begin{equation}\label{section1_nonlocal_operator}
\mathcal{L}u(\textbf{x})=\int_{\Omega}J(\textbf{x}-\textbf{y})[u(\textbf{x})-u(\textbf{y})]d\textbf{y},\quad \forall \textbf{x}\in\Omega\subseteq\mathbb{R}^n.
\end{equation}

The main goal of this paper is to construct accurate and efficient linear algorithms for the general nonlocal Cahn-Hilliard equation with general nonlinear potential and prove the unconditional energy stability for its semi-discrete schemes carefully and rigorously. In addition, considering the huge computational work and memory requirement in solving the linear system, we analyze the structure of the stiffness matrix and seek some effective fast solution method to reduce the computational work and memory requirement. By applying four transformation operators $\mathcal{A}_1$, $\mathcal{A}_2$, $\mathcal{A}_3$, $\mathcal{A}_4$, we transform the stiffness matrix into a block-Toeplitz-Toeplitz-block (BTTB) matrix. Then, a fast solution technique which is based on a fast Fourier transform is presented to solve a new linear system with BTTB stiffness matrix. The overall computational cost of the fast conjugate gradient method is $O(N$log$^2N)$, since the number of iterations is $O($log$N)$ where $N$ is the number of unknowns. What we need to focus is that if one uses the Gaussian elimination method straightforwardly to this linear system, then it requires $O(N^3)$ complexity. In addition, since $N\times N$ BTTB matrix is determined by only $2N-1$ entries rather than $N^2$ entries, the fast solver will make memory requirement $O(N)$ instead of $O(N^{2})$. Finally, various 2D numerical simulations are demonstrated to verify the accuracy and efficiency of our proposed schemes.

The paper is organized as follows. In Sect.2, we provide the nonlocal Cahn-Hilliard model with general nonlinear potential by energy variational approach and present some notations. In Sect.3, the linear, first and second order numerical scalar auxiliary variable approaches to construct unconditionally energy stable schemes for the nonlocal Cahn-Hilliard model are considered. In Sect.4, we analyse the structure of the stiffness matrix and seek some effective fast solution methods to reduce the computational work and memory requirement. In Sect.5, various 2D numerical simulations are demonstrated to verify the accuracy and efficiency of our proposed schemes.

Throughout the paper, we use $C$, with or without subscript, to denote a positive constant, which could have different values at different appearances.

\section{The nonlocal Cahn-Hilliard model and relative notations}
In this section, we introduce the nonlocal Cahn-Hilliard model with general nonlinear potential by energy variational approach and present some notations which will be used in the later analysis.

First, the inner product and norm of $L^2(\Omega)$ are defined by
\begin{equation}\label{section1_e_inner product}
(u,v)=\int_{\Omega}uvd\textbf{x},\quad \|u\|_{2}=(u,u)^{\frac{1}{2}},\quad \forall\ u,\ v\in L^2(\Omega).
\end{equation}

It is easy to obtain that
\begin{equation}\label{section1_nonlocal_inner}
\left(\mathcal{L}u(\textbf{x}),u(\textbf{x})\right)=2\int_{\Omega}\int_{\Omega}J(\textbf{x}-\textbf{y})[u(\textbf{x})-u(\textbf{y})]^2d\textbf{y}d\textbf{x}\geq0.
\end{equation}

Next, we give a brief introduction about how the nonlocal Cahn-Hilliard model is resulted from the energetic variation of the energy functional (\ref{section1_e_energy functional}). Denoting its variational derivative as $\mu=\frac{\delta E}{\delta \phi}$, the general form of the gradient flow model can be written as \cite{shen2017new}
\begin{equation}\label{section1_gradient_flow}
\frac{\partial \phi}{\partial t}=M\mathcal{G}\mu, \quad(\textbf{x},t)\in\Omega\times Q,
\end{equation}
where $Q=(0,T]$, $M$ is the mobility constant, $\mu$ is the chemical potential, and $f(\phi)=F^{\prime}(\phi)$. The initial condition is $\phi|_{t=0}=\phi_0$. The equation (\ref{section1_gradient_flow}) will be Cahn-Hilliard type system if $\mathcal{G}=\Delta$. The equation is supplemented with the following boundary condition: periodic, or $\frac{\partial \phi}{\partial \textbf{n}}=\frac{\partial \mu}{\partial \textbf{n}}=0$, where $\textbf{n}$ is the unit outward normal vector on the boundary $\partial\Omega$.

Noting that for Cahn-Hilliard model, the operator $\mathcal{G}=-\mathcal{L}$ is non-positive, the free energy $E(\phi)$ in equation (\ref{section1_e_energy functional}) is non-increasing,
\begin{equation}\label{section1_eneegy_decay}
\frac{dE[\phi(t)]}{dt}=\frac{\delta E(\phi)}{\delta\phi}\cdot\frac{\partial\phi}{\partial t}=-(\mu,M\mathcal{L}\mu)\leq0.
\end{equation}

By some simple calculations, we obtain
\begin{equation}\label{section1_eneegy_compute}
\aligned
\frac{dE[\phi(t)]}{dt}
&=\int_{\Omega}\left(F'(\phi)\phi_t+\frac{\epsilon^2}{2}\int_\Omega J(\textbf{x}-\textbf{y})\left[\phi(\textbf{x})-\phi(\textbf{y})\right]\left[\phi_t(\textbf{x})-\phi_t(\textbf{y})\right]d\textbf{y}\right)d\textbf{x}\\
&=\int_{\Omega}F'(\phi)\phi_tdx+\frac{\epsilon^2}{2}\int_{\Omega}\int_\Omega J(\textbf{x}-\textbf{y})\left[\phi(\textbf{x})-\phi(\textbf{y})\right]\phi_t(\textbf{x})d\textbf{y}d\textbf{x}\\
&\quad-\frac{\epsilon^2}{2}\int_{\Omega}\int_\Omega J(\textbf{x}-\textbf{y})\left[\phi(\textbf{x})-\phi(\textbf{y})\right]\phi_t(\textbf{y})d\textbf{y}d\textbf{x}.
\endaligned
\end{equation}

Using the condition $(ii)$ of kernel $J$ and variable substitution, we obtain
\begin{equation}\label{section1_eneegy_compute2}
\aligned
&\int_{\Omega}\int_\Omega J(\textbf{x}-\textbf{y})\left[\phi(\textbf{x})-\phi(\textbf{y})\right]\phi_t(\textbf{y})d\textbf{y}d\textbf{x}\\
&=-\int_{\Omega}\int_\Omega J(\textbf{x}-\textbf{y})\left[\phi(\textbf{x})-\phi(\textbf{y})\right]\phi_t(\textbf{x})d\textbf{y}d\textbf{x}.
\endaligned
\end{equation}

Combining equations $(\ref{section1_eneegy_compute})$ and equation $(\ref{section1_eneegy_compute2})$, we obtain
\begin{equation}\label{section1_eneegy_compute3}
\aligned
\frac{dE[\phi(t)]}{dt}
&=\int_{\Omega}\left(F'(\phi)+\epsilon^2\int_\Omega J(\textbf{x}-\textbf{y})\left[\phi(\textbf{x})-\phi(\textbf{y})\right]d\textbf{y}\right)\phi_t(\textbf{x})d\textbf{x}.
\endaligned
\end{equation}

Then, combining the equations (\ref{section1_nonlocal_operator}), (\ref{section1_gradient_flow}), (\ref{section1_eneegy_decay}), (\ref{section1_eneegy_compute3}) with $\mu=\frac{\delta E}{\delta \phi}$ and $\mathcal{G}=-\mathcal{L}$, we obtain the nonlocal Cahn-Hilliard model with general nonlinear potential:
\begin{equation}\label{section1_e_model2}
  \left\{
   \begin{array}{rlr}
   \displaystyle \frac{\partial \phi}{\partial t}&=-M\mathcal{L}\mu,                   & (\textbf{x},t)\in\Omega\times J,\\
                                              \mu&=\displaystyle F'(\phi)+\epsilon^2\mathcal{L}\phi(\textbf{x}), & (\textbf{x},t)\in\Omega\times J.
   \end{array}
   \right.
\end{equation}

\section{The scalar auxiliary variable semi-implicit schemes for nonlocal Cahn-Hilliard model}
In this section, we execute and analyze linear, first and second order (in time) numerical scalar auxiliary variable (SAV) approaches to construct unconditionally energy stable schemes for the nonlocal Cahn-Hilliard model with general nonlinear potential. The SAV approach is proposed for a large class of gradient flows that describes energy dissipative physical systems. It leads to numerical schemes that enjoy linear second-order unconditionally energy stability, which is very significant to solve the stiffness issue from the thin interface. For the SAV scheme, we only assume $E_1(\phi)=\int_{\Omega}F(\phi)d\textbf{x}$ to be bounded from below, i.e., $E_1(\phi)\geq -C_0$. Similar to \cite{shen2018scalar}, we introduce a scalar auxiliary variable (SAV):
\begin{equation*}
r(t)=\sqrt{E_1+C_0}.
\end{equation*}

Then, the nonlocal Cahn-Hilliard model (\ref{section1_e_model2}) can be transformed into the following formulation:
\begin{equation}\label{section2_e_SAV approach1}
  \left\{
   \begin{array}{rl}
   \displaystyle \frac{\partial \phi}{\partial t}&=-M\mathcal{L}\mu,\\
   \displaystyle \mu&\displaystyle=\epsilon^2\mathcal{L}\phi(\textbf{x})+\frac{r}{\sqrt{E_1(\phi)+C_0}}F'(\phi),\\
    \displaystyle \frac{d r}{d t}&\displaystyle=\frac{1}{2\sqrt{E_1(\phi)+C_0}}\int_{\Omega}F'(\phi)\frac{\partial \phi}{\partial t}d\textbf{x}.
   \end{array}
   \right.
  \end{equation}

Taking the inner products of the above equations with $\mu$, $\phi_t$ and $2r$, respectively, one can obtain the modified energy dissipation law:
\begin{equation*}
\frac{\partial}{\partial t}\left[\epsilon^2(\mathcal{L}\phi,\phi)+|r|^2\right]=-M(\mathcal{L}\mu,\mu),
\end{equation*}
where $
(\mathcal{L}\phi,\phi)=2\int_\Omega \int_\Omega J(\textbf{x}-\textbf{y})\left[\phi(\textbf{x})-\phi(\textbf{y})\right]^2d\textbf{y}d\textbf{x}\geq0$.

Let $N>0$ be a positive integer and set
$$\Delta t=T/N,\quad t^n=n\Delta t,\quad \text{for}\quad n\leq N.$$

A semi-implicit first order SAV scheme for (\ref{section2_e_SAV approach1}) reads as
\begin{empheq}[left=\empheqlbrace]{alignat=2}
&\displaystyle \frac{\phi^{n+1}-\phi^{n}}{\Delta t}=-M\mathcal{L}\mu^{n+1},\label{section2_e_SAV approach2}\\
&\displaystyle \mu^{n+1}=\epsilon^2\mathcal{L}\phi^{n+1}+\frac{r^{n+1}}{\sqrt{E_1(\tilde{\phi}^{n+1})+C_0}}F'(\tilde{\phi}^{n+1}), \label{section2_e_SAV approach2*}\\
&\displaystyle \frac{r^{n+1}-r^{n}}{\Delta t}=\int_{\Omega}\frac{F'(\tilde{\phi}^{n+1})}{2\sqrt{E_1(\tilde{\phi}^{n+1})+C_0}}\frac{\phi^{n+1}-\phi^{n}}{\Delta t}d\textbf{x}, \label{section2_e_SAV approach2**}
\end{empheq}
where $\tilde{\phi}^{n+1}$ is any explicit $O(\Delta t)$ approximation for $\phi(t^{n+1})$, which can be flexible according to the problem. For instance, we may use an extrapolation as follows
\begin{equation}\label{section2_e_SAV approach3}
\aligned
\tilde{\phi}^{n+1}=2\phi^n-\phi^{n-1}.
\endaligned
\end{equation}
Besides, we can also use the following simple first-order scheme to obtain it:
\begin{equation}\label{section2_e_SAV approach4}
\aligned
\frac{\tilde{\phi}^{n+1}-\phi^n}{\Delta t}=-M\mathcal{L}(\epsilon^2\mathcal{L}\tilde{\phi}^{n+1}+F'(\phi^n)).
\endaligned
\end{equation}

\begin{theorem}\label{section2_th_SAV1}
The scheme (\ref{section2_e_SAV approach2})-(\ref{section2_e_SAV approach2**}) for Cahn-Hilliard type system is unconditionally energy stable in the sense that
\begin{equation}\label{section2_e_SAV1}
\aligned
\frac{\epsilon^2}{2}(\mathcal{L}\phi^{n+1},\phi^{n+1})+|r^{n+1}|^2\leq\frac{\epsilon^2}{2}(\mathcal{L}\phi^{n},\phi^{n})+|r^n|^2.
\endaligned
\end{equation}
\end{theorem}
\begin{proof}
By taking the inner products with $\Delta t\mu^{n+1}$, $\phi^{n+1}-\phi^n$, $2\Delta tr^{n+1}$ for equations (\ref{section2_e_SAV approach2}), (\ref{section2_e_SAV approach2*}), (\ref{section2_e_SAV approach2**}) respectively and some simple calculations, we obtain
\begin{equation}\label{section2_e_SAV approach5}
\aligned
\epsilon^2(\mathcal{L}\phi^{n+1},\phi^{n+1}-\phi^n)+2(r^{n+1},r^{n+1}-r^n)=-M\Delta t(\mathcal{L}\mu^{n+1},\mu^{n+1})\leq0.
\endaligned
\end{equation}

Using the identity $\textbf{x}\cdot(\textbf{x}-\textbf{y})=\frac{1}{2}|\textbf{x}|^2-\frac{1}{2}|\textbf{y}|^2+\frac{1}{2}|\textbf{x}-\textbf{y}|^2$, the equation (\ref{section2_e_SAV approach5}) can be transformed as follows:
\begin{equation}\label{section2_e_SAV approach6}
\aligned
&\frac{\epsilon^2}{2}(\mathcal{L}\phi^{n+1},\phi^{n+1})-\frac{\epsilon^2}{2}(\mathcal{L}\phi^{n},\phi^{n})+\frac{\epsilon^2}{2}(\mathcal{L}(\phi^{n+1}-\phi^n),\phi^{n+1}-\phi^n)\\
&\quad+|r^{n+1}|^2-|r^n|^2+|r^{n+1}-r^n|^2=-M\Delta t(\mathcal{L}\mu^{n+1},\mu^{n+1})\leq0,
\endaligned
\end{equation}
which states that $\frac{\epsilon^2}{2}(\mathcal{L}\phi^{n+1},\phi^{n+1})+|r^{n+1}|^2\leq\frac{\epsilon^2}{2}(\mathcal{L}\phi^{n},\phi^{n})+|r^n|^2$.
\end{proof}

A semi-implicit second order SAV/BDF scheme for (\ref{section2_e_SAV approach1}) reads as
\begin{empheq}[left=\empheqlbrace]{alignat=2}
&\displaystyle \frac{3\phi^{n+1}-4\phi^{n}+\phi^{n-1}}{2\Delta t}=-M\mathcal{L}\mu^{n+1},\label{section2_e_SAV approach7}\\
&\displaystyle \mu^{n+1}=\epsilon^2\mathcal{L}\phi^{n+1}+\frac{r^{n+1}}{\sqrt{E_1(\tilde{\phi}^{n+\frac{1}{2}})+C_0}}F'(\tilde{\phi}^{n+\frac{1}{2}}), \label{section2_e_SAV approach7*}\\
&\displaystyle \frac{3r^{n+1}-4r^{n}+r^{n-1}}{2\Delta t}=\int_{\Omega}\frac{F'(\tilde{\phi}^{n+\frac{1}{2}})}{2\sqrt{E_1(\tilde{\phi}^{n+\frac{1}{2}})+C_0}}\frac{3\phi^{n+1}-4\phi^{n}+\phi^{n-1}}{2\Delta t}d\textbf{x}, \label{section2_e_SAV approach7**}
\end{empheq}
where $\tilde{\phi}^{n+\frac{1}{2}}$ is any explicit $O(\Delta t^2)$ approximation for $\phi(t^{n+1})$, which can be flexible according to the problem.
\begin{theorem}\label{section2_th_SAV2}
The scheme (\ref{section2_e_SAV approach7})-(\ref{section2_e_SAV approach7**}) for Cahn-Hilliard type system is unconditionally energy stable in the sense that
\begin{equation}\label{section2_e_SAV2}
\aligned
&\frac{\epsilon^2}{2}(\mathcal{L}\phi^{n+1},\phi^{n+1})+\frac{\epsilon^2}{2}(\mathcal{L}(2\phi^{n+1}-\phi^n),2\phi^{n+1}-\phi^n)+|r^{n+1}|^2+|2r^{n+1}-r^n|^2\\
&\leq\frac{\epsilon^2}{2}(\mathcal{L}\phi^{n},\phi^{n})+\frac{\epsilon^2}{2}(\mathcal{L}(2\phi^{n}-\phi^{n-1}),2\phi^{n}-\phi^{n-1})+|r^{n}|^2+|2r^{n}-r^{n-1}|^2.
\endaligned
\end{equation}
\end{theorem}
\begin{proof}
By taking the inner products with $2\Delta t\mu^{n+1}$, $3\phi^{n+1}-4\phi^{n}+\phi^{n-1}$, $2\Delta tr^{n+1}$ for (\ref{section2_e_SAV approach7}), (\ref{section2_e_SAV approach7*}), (\ref{section2_e_SAV approach7**}) respectively and some simple calculations, we can obtain
\begin{equation}\label{section2_e_SAV approach8}
\aligned
&\epsilon^2(\mathcal{L}\phi^{n+1},3\phi^{n+1}-4\phi^{n}+\phi^{n-1})+2(r^{n+1},r^{n+1}-r^n)\\
&=-2M\Delta t(\mathcal{L}\mu^{n+1},\mu^{n+1})\leq0.
\endaligned
\end{equation}

Using the identity $$2(\textbf{x},3\textbf{x}-4\textbf{y}+\textbf{z})=|\textbf{x}|^2+|\textbf{x}-\textbf{y}|^2+|\textbf{x}-2\textbf{y}+\textbf{z}|^2-|\textbf{y}|^2-|2\textbf{y}-\textbf{z}|^2,$$ the equation (\ref{section2_e_SAV approach8}) can be transformed as follows:
\begin{equation}\label{section2_e_SAV approach9}
\aligned
&\frac{\epsilon^2}{2}(\mathcal{L}\phi^{n+1},\phi^{n+1})+\frac{\epsilon^2}{2}\left(\mathcal{L}(2\phi^{n+1}-\phi^n),2\phi^{n+1}-\phi^n\right)\\
&\quad+\frac{\epsilon^2}{2}\left(\mathcal{L}(\phi^{n+1}-2\phi^n+\phi^{n-1}),\phi^{n+1}-2\phi^n+\phi^{n+1}\right)\\
&\quad-\frac{\epsilon^2}{2}(\mathcal{L}\phi^{n},\phi^{n})-\frac{\epsilon^2}{2}\left(\mathcal{L}(2\phi^{n}-\phi^{n-1}),2\phi^{n}-\phi^{n-1}\right)\\
&\quad+|r^{n+1}|^2+|2r^{n+1}-r^{n}|^2+|r^{n+1}-2r^n+r^{n-1}|^2-|r^{n}|^2-|2r^{n}-r^{n-1}|^2\\
&=-2M\Delta t(\mathcal{L}\mu^{n+1},\mu^{n+1})\leq0,
\endaligned
\end{equation}
which completes the proof.
\end{proof}
\section{The SAV finite difference scheme for nonlocal Cahn-Hilliard model}
In this section, we consider finite difference discretization for nonlocal Cahn-Hilliard model for some spatial operators in the two dimensional space with $\Omega=(-L,L)\times(-R,R)$. We give the linear second order (in time) SAV/BDF finite difference scheme for nonlocal Cahn-Hilliard model. The first order SAV fully discrete scheme can be obtained straightforwardly.

In order to discretize the system (\ref{section2_e_SAV approach7})-(\ref{section2_e_SAV approach7**}), we define $\Omega_h=\{(x_i,y_j)|x_i=-L+ih_x,~y_j=-R+jh_y,~h_x=2L/M_x,~h_y=2R/M_y,~0\leq i\leq M_x,~0\leq j\leq M_y\}$  to be a uniform mesh of the domain $\Omega$.

From the analysis of the nonlocal operator $\mathcal{L}$ in \cite{du2018stabilized}, one can see that
\begin{equation*}
\mathcal{L}\phi=(J\ast1)\phi-J\ast\phi.
\end{equation*}

Then, for any $v$, $\mathcal{L}v$ can be discreted at $(t^{n},x_i,y_j),$ $(0\leq n\leq N,~0\leq i\leq M_x,~0\leq j\leq M_y)$ as follows:
\begin{equation}\label{section3_e_SAV1}
(\mathcal{L}_hv)_{i,j}^n=(J\ast1)_{i,j}v_{i,j}^n-(J\ast v)_{i,j}^n,
\end{equation}
where
\begin{equation*}
\aligned
(J\ast1)_{i,j}v_{i,j}^n
&=h_xh_y\left[\sum\limits_{m_1=1}^{M_x-1}\sum\limits_{m_2=1}^{M_y-1}J(x_{m_1}-x_i,y_{m_2}-y_j)\right.\\
&\quad+\frac{1}{2}\sum\limits_{m_1=1}^{M_x-1}\left(J(x_{m_1}-x_i,y_0-y_j)+J(x_{m_1}-x_i,y_{M_y}-y_j)\right)\\
&\quad+\frac{1}{2}\sum\limits_{m_2=1}^{M_y-1}\left(J(x_0-x_i,y_{m_2}-y_j)+J(x_{M_x}-x_i,y_{m_2}-y_j)\right)\\
&\quad+\frac{1}{4}\left(J(x_0-x_i,y_0-y_j)+J(x_{M_x}-x_i,y_0-y_j)\right)\\
&\quad\left.+\frac{1}{4}\left(J(x_0-x_i,y_{M_y}-y_j)+J(x_{M_x}-x_i,y_{M_y}-y_j)\right)\right]v_{i,j}^n,
\endaligned
\end{equation*}
and
\begin{equation*}
\aligned
(J\ast v)_{i,j}^n
&=h_xh_y\left[\sum\limits_{m_1=1}^{M_x-1}\sum\limits_{m_2=1}^{M_y-1}J(x_{m_1}-x_i,y_{m_2}-y_j)v_{m_1,m_2}^n\right.\\
&\quad+\frac{1}{2}\sum\limits_{m_1=1}^{M_x-1}\left(J(x_{m_1}-x_i,y_0-y_j)v_{m_1,0}^n+J(x_{m_1}-x_i,y_{M_y}-y_j)v_{m_1,M_y}^n\right)\\
&\quad+\frac{1}{2}\sum\limits_{m_2=1}^{M_y-1}\left(J(x_0-x_i,y_{m_2}-y_j)v_{0,m_2}^n+J(x_{M_x}-x_i,y_{m_2}-y_j)v_{M_x,m_2}^n\right)\\
&\quad+\frac{1}{4}\left(J(x_0-x_i,y_0-y_j)v_{0,0}^n+J(x_{M_x}-x_i,y_0-y_j)v_{M_x,0}^n\right)\\
&\quad\left.+\frac{1}{4}\left(J(x_0-x_i,y_{M_y}-y_j)v_{0,M_y}^n+J(x_{M_x}-x_i,y_{M_y}-y_j)v_{M_x,M_y}^n\right)\right].
\endaligned
\end{equation*}

Combining the semi-implicit scheme (\ref{section2_e_SAV approach7})-(\ref{section2_e_SAV approach7}) with equation (\ref{section3_e_SAV1}), we obtain the finite difference discretization for nonlocal Cahn-Hilliard model (\ref{section2_e_SAV approach1}) as follows:
\begin{empheq}[left=\empheqlbrace]{alignat=2}
&\displaystyle \frac{3\phi_{i,j}^{n+1}-4\phi_{i,j}^{n}+\phi_{i,j}^{n-1}}{2\Delta t}=-M(\mathcal{L}_h\mu)_{i,j}^{n+1},\label{section3_e_SAV approach1}\\
&\displaystyle \mu_{i,j}^{n+1}=\epsilon^2(\mathcal{L}_h\phi)_{i,j}^{n+1}+\frac{r^{n+1}}{\sqrt{E_1(\tilde{\phi}^{n+\frac{1}{2}})+C_0}}F'(\tilde{\phi}_{i,j}^{n+\frac{1}{2}}), \label{section3_e_SAV approach1*}\\
&\displaystyle \frac{3r^{n+1}-4r^{n}+r^{n-1}}{2\Delta t}=\int_{\Omega}\frac{F'(\tilde{\phi}^{n+\frac{1}{2}})}{2\sqrt{E_1(\tilde{\phi}^{n+\frac{1}{2}})+C_0}}\frac{3\phi^{n+1}-4\phi^{n}+\phi^{n-1}}{2\Delta t}d\textbf{x}. \label{section3_e_SAV approach1**}
\end{empheq}

\begin{lemma}\label{le2}[Solvability of the SAV finite difference scheme]
For any $\Delta t$, $h_x$, $h_y>0$, the scheme (\ref{section3_e_SAV approach1})-(\ref{section3_e_SAV approach1**}) has a unique solution.
\end{lemma}
\begin{proof}
Denote
\begin{equation*}
\eta^{n+1}=\frac{F'(\tilde{\phi}^{n+\frac{1}{2}})}{\sqrt{E_1(\tilde{\phi}^{n+\frac{1}{2}})+C_0}}.
\end{equation*}
Then, the scheme (\ref{section3_e_SAV approach1})-(\ref{section3_e_SAV approach1**}) can be rewritten as the following formulation:
\begin{equation*}
\aligned
&\left(I_h+\frac{2}{3}M\Delta t\epsilon^2\mathcal{L}_h^2\right)\Phi^{n+1}+\frac{1}{6}M\Delta t(\eta^{n+1},\Phi^{n+1})\mathcal{L}_h\eta^{n+1}\\
&=\frac{4}{3}I_h\Phi^n-\frac{1}{3}I_h\Phi^{n-1}-\frac{1}{3}M\Delta t\left[\frac{4}{3}r^n-\frac{1}{3}r^{n-1}-\frac{2}{3}(\eta^{n+1},\Phi^n)\right.\\
&\quad\left.+\frac{1}{6}(\eta^{n+1},\Phi^{n-1})\right]\mathcal{L}_h\eta^{n+1}.
\endaligned
\end{equation*}

Obviously, the stiffness matrix $A=\left(I_h+\frac{2}{3}M\Delta t\epsilon^2\mathcal{L}_h^2\right)$ is positive definite. From the above equation, we observe that $(\eta^{n+1},\Phi^{n+1})$ needs to be computed first. Multiplying the above equation with $A^{-1}$, and taking the inner product with $\eta^{n+1}$, we obtain
\begin{equation*}
\aligned
&(1+\frac{1}{6}M\Delta t\theta)(\eta^{n+1},\Phi^{n+1})\\
&=\frac{4}{3}(A^{-1}\Phi^n,\eta^{n+1})-\frac{1}{3}(A^{-1}\Phi^{n-1},\eta^{n+1})\\
&\quad-\frac{1}{3}M\Delta t\theta\left[\frac{4}{3}r^n-\frac{1}{3}r^{n-1}-\frac{2}{3}(\eta^{n+1},\Phi^n)+\frac{1}{6}(\eta^{n+1},\Phi^{n-1})\right],
\endaligned
\end{equation*}
where $\theta=(A^{-1}\mathcal{L}_h\eta^{n+1},\eta^{n+1})\geq0$.

Noting that $(1+\frac{1}{6}M\Delta t\theta)\neq0$ and the stiffness matrix $A$ is positive definite, one can conclude that the scheme (\ref{section3_e_SAV approach1})-(\ref{section3_e_SAV approach1**}) has a unique solution.
\end{proof}
\section{The fast solution method}
From the discrete formulation (\ref{section3_e_SAV1}), one can see that the nonlocal diffusion term will
lead the stiffness matrix to be an almost full matrix which requires huge computational work and large memory. In this case, the fast solution method for solving derived linear system will become very important and necessary. In this section, we will analyse the structure of the stiffness matrix and seek some effective fast solution method to reduce the computational work and memory requirement. This fast solution technique is based on a fast Fourier transform and depends on the special structure of coefficient matrices.

Without loss of generality, we assume $\Omega=[-L,L]\times[-L,L]$ and the partition is uniform and satisfies $h=h_x=h_y$. For any $Au=f$ in solving the discrete formulation \eqref{section3_e_SAV approach1}-\eqref{section3_e_SAV approach1**}, we all use conjugate gradient method to obtain the solution of linear system: Let $u_{0}$ be an initial guess. Then, compute $r_{0}=f-Au_{0}$, $d_{1}=r_{0}$ and
\begin{flalign*}
\begin{split}
\hspace{10mm}%
&\omega_{1}=r^{T}_{0}r_{0}/d^{T}_{1}Ad_{1}\\
\hspace{10mm}%
&u_{1}=u_{0}+\omega_{1}d_{1}\\
\hspace{10mm}%
&r_{1}=r_{0}-\omega_{1}Ad_{1}\\
for\ &k=2,3,\ldots\\
\hspace{10mm}%
&\gamma_{k}=r^{T}_{k-1}r_{k-1}/r^{T}_{k-2}r_{k-2}\\
\hspace{10mm}%
&d_{k}=r_{k-1}+\gamma_{k}d_{k-1}\\
\hspace{10mm}%
&\omega_{k}=r^{T}_{k-1}r_{k-1}/d^{T}_{k}Ad_{k}\\
\hspace{10mm}%
&u_{k}=u_{k-1}+\omega_{k}d_{k}\\
\hspace{10mm}%
&r_{k}=r_{k-1}-\omega_{k}Ad_{k}\\
\hspace{10mm}%
&\text{Check\ for\ convergence,\ continue\ if\ necessary}\\
end&\\
&u=u_{k}.
\end{split}&
\end{flalign*}

From the algorithm of conjugate gradient method, one can see that for reducing the huge computational work and memory requirement, we only need to find fast and efficient procedure to accelerate the matrix-vector multiplication $Ad$ for any vector $d$ and store $A$ efficiently.

Based on the above analysis, we note that the stiffness matrix $A=\left(I_h+\frac{2}{3}M\Delta t\epsilon^2\mathcal{L}_h^2\right)$. So, we only need to analyze the structure of the matrix-vector multiplication $\mathcal{L}_hd$.

First, we can rewrite the discrete formulation $(J\ast v)_{i,j}$ as follows
\begin{equation}\label{section4_e1}
\aligned
(J\ast v)_{i,j}
&=\frac{1}{4}h^2\sum\limits_{m_1=0}^{M-1}\sum\limits_{m_2=0}^{M-1}J(x_{m_1}-x_i,y_{m_2}-y_j)v_{m_1,m_2}\\
&+\frac{1}{4}h^2\sum\limits_{m_1=1}^{M}  \sum\limits_{m_2=0}^{M-1}J(x_{m_1}-x_i,y_{m_2}-y_j)v_{m_1,m_2}\\
&+\frac{1}{4}h^2\sum\limits_{m_1=0}^{M-1}\sum\limits_{m_2=1}^{M}  J(x_{m_1}-x_i,y_{m_2}-y_j)v_{m_1,m_2}\\
&+\frac{1}{4}h^2\sum\limits_{m_1=1}^{M}  \sum\limits_{m_2=1}^{M}  J(x_{m_1}-x_i,y_{m_2}-y_j)v_{m_1,m_2}.
\endaligned
\end{equation}

Define $\gamma_{i,j}=(J\ast v)_{i,j}$, ${\bf{\Lambda}}_i=(\gamma_{i,0},\gamma_{i,1},\cdots,\gamma_{i,M})$ and ${\bf{\Lambda}}=({\bf{\Lambda}}_1,{\bf{\Lambda}}_2,\cdots,{\bf{\Lambda}}_M)^T$. Then, define four transformation operators $\mathcal{A}_1$, $\mathcal{A}_2$, $\mathcal{A}_3$, $\mathcal{A}_4$. For any $(M+1)^2$-vector $v=(v_0,v_1,\ldots,v_{M})^T,$ $v_j=(v_{0,j},v_{1,j},\ldots,x_{M,j})$, the operators $\mathcal{A}_i$, $(i=1,2,3,4)$ satisfy:
\begin{equation*}
   \mathcal{A}_iv=f_i,\quad i=1,2,3,4,
\end{equation*}
where for $0\leq j\leq M-1$ and $1\leq k\leq M$,
\begin{equation}\label{section4_e2}
   \begin{array}{rlrl}
   f_1&=(\alpha_0,\alpha_1,\ldots,\alpha_{M-1},\textbf{0})^T, & \alpha_j&=(v_{0,j},v_{1,j},\ldots,v_{M-1,j},0),\\
   f_2&=(\textbf{0},\beta_1, \beta_2,\ldots,\beta_{M})^T,     & \beta_k &=(v_{0,k},v_{1,k},\ldots,v_{M-1,k},0),\\
   f_3&=(\zeta_0,\zeta_1,\ldots,\zeta_{M-1},\textbf{0})^T,    & \zeta_j &=(0,v_{1,j},v_{2,j},\ldots,v_{M,j}),\\
   f_4&=(\textbf{0},\theta_1,\theta_2,\ldots,\theta_{M})^T,   & \theta_k&=(0,v_{1,k},v_{2,k},\ldots,v_{M,k}),
   \end{array}
\end{equation}

Define the following four vectors $\alpha_M$, $\beta_0$, $\zeta_M$ and $\theta_0$ to be all zero $(1\times M)$-vectors. Then, the equation (\ref{section4_e1}) can be rewritten as the following formulation
\begin{equation*}
\aligned
(J\ast v)_{i,j}
&=\frac{1}{4}h^2\sum\limits_{m_1=0}^{M}\sum\limits_{m_2=0}^{M}J(x_{m_1}-x_i,y_{m_2}-y_j)\alpha_{m_2}(m_1)\\
&+\frac{1}{4}h^2\sum\limits_{m_1=0}^{M}\sum\limits_{m_2=0}^{M}J(x_{m_1}-x_i,y_{m_2}-y_j)\beta_{m_2}(m_1)\\
&+\frac{1}{4}h^2\sum\limits_{m_1=0}^{M}\sum\limits_{m_2=0}^{M}J(x_{m_1}-x_i,y_{m_2}-y_j)\zeta_{m_2}(m_1)\\
&+\frac{1}{4}h^2\sum\limits_{m_1=0}^{M}\sum\limits_{m_2=0}^{M}J(x_{m_1}-x_i,y_{m_2}-y_j)\theta_{m_2}(m_1)\\
&=\frac{1}{4}h^2\sum\limits_{m_1=0}^{M}\sum\limits_{m_2=0}^{M}J(x_{m_1}-x_i,y_{m_2}-y_j)\\
&\quad\times\left[\alpha_{m_2}(m_1)+\beta_{m_2}(m_1)+\zeta_{m_2}(m_1)+\theta_{m_2}(m_1)\right].
\endaligned
\end{equation*}

By the above equation, we can compute the vector ${\bf{\Lambda}}$ by the following equation:
\begin{equation}\label{section4_e3}
\aligned
{\bf{\Lambda}}=\frac{1}{4}h^2\textbf{B}(f_1+f_2+f_3+f_4).
\endaligned
\end{equation}

Then, for any $(M+1)^2$-vector $v$, we obtain the matrix-vector multiplication $\mathcal{L}_hv$ by
\begin{equation*}
\aligned
\mathcal{L}_hv=\frac{1}{4}h^2\textbf{B}(f_1+f_2+f_3+f_4)-\frac{1}{4}h^2\left[\textbf{B}(\mathcal{A}_1+\mathcal{A}_2+\mathcal{A}_3+\mathcal{A}_4)\mathcal{I}\right]\cdot v.
\endaligned
\end{equation*}

Combining equation (\ref{section4_e1}) and equation (\ref{section4_e3}), and noting the four vectors in (\ref{section4_e2}), the matrix $\textbf{B}$ can be written as follows:
\begin{flalign*}
\textbf{B}=\left(
\begin{array}{cccccccc}
\textbf{S}_{0}   &\textbf{S}_{1}   &\textbf{S}_{2}   &\cdots  &\textbf{S}_{M-2} &\textbf{S}_{M-1} &S_{M}\\
\textbf{S}_{1}   &\textbf{S}_{0}   &\textbf{S}_{1}   &\cdots  &\textbf{S}_{M-3} &\textbf{S}_{M-2} &\textbf{S}_{M-1}\\
\textbf{S}_{2}   &\textbf{S}_{1}   &\textbf{S}_{0}   &\textbf{S}_{1}   &\vdots  &\vdots  &\vdots\\
\vdots  &\vdots  &\ddots  &\ddots  &\ddots  &\vdots  &\vdots\\
\textbf{S}_{M-2} &\textbf{S}_{M-3} &\cdots  &\cdots  &\textbf{S}_{0}   &\textbf{S}_{1}   &\textbf{S}_{2}\\
\textbf{S}_{M-1} &\textbf{S}_{M-2} &\textbf{S}_{M-3} &\cdots  &\textbf{S}_{1}   &\textbf{S}_{0}   &\textbf{S}_{1}\\
\textbf{S}_{M}   &\textbf{S}_{M-1} &\textbf{S}_{M-2} &\textbf{S}_{M-3} &\cdots  &\textbf{S}_{1}   &\textbf{S}_{0}\\
\end{array}
\right).
\end{flalign*}

Define $x_{i,j}=J(ih,jh)$. By simple calculation, the block matrix $\textbf{S}_j$ can be expressed as the following Toeplitz formulation:
\begin{flalign*}
\textbf{S}_j=\left(
\begin{array}{ccccccc}
x_{0,j}     &x_{1,j}         &\cdots      &x_{M-1,j}   &x_{M,j}\\
x_{1,j}     &x_{0,j}         &\cdots      &x_{M-2,j}   &x_{M-1,j}\\
\vdots      &\vdots          &\ddots      &\vdots      &\vdots\\
x_{M-1,j}   &x_{M-2,j}       &\cdots      &x_{0,j}     &x_{1,j}\\
x_{M,j}     &x_{M-1,j}       &x_{M-2,j}   &\cdots      &x_{0,j}\\
\end{array}
\right).
\end{flalign*}

Therefore, we note that the matrix $\textbf{B}$ is a Block-Toeplitz-Toeplitz-Block (BTTB) matrix. Then, we can use fast Fourier transform (FFT) method to evaluate $\textbf{B}G$ for $G=f_1+f_2+f_3+f_4$ which can reduce the computational work and memory requirement effectively.

Firstly, the block Toeplitz matrix $\textbf{S}_j$ can be embedded into a $(2M+2)\times(2M+2)$ circulant matrix:
\begin{flalign*}
\textbf{C}_j=\left(
\begin{array}{cccccc}
\textbf{S}_j&\textbf{Q}_j\\
\textbf{Q}_j&\textbf{S}_j
\end{array}
\right),
\qquad \textbf{Q}_j=\left(
\begin{array}{ccccccc}
0         &x_{M,j}     &\cdots      &x_{2,j}   &x_{1,j}\\
x_{M,j}   &0           &\cdots      &x_{3,j}   &x_{2,j}\\
\vdots    &\vdots      &\ddots      &\vdots    &\vdots\\
x_{2,j}   &x_{3,j}     &\cdots      &0         &x_{M,j}\\
x_{1,j}   &x_{2,j}     &x_{3,j}     &\cdots    &0\\
\end{array}
\right).
\end{flalign*}

Then, replacing the block matrix $\textbf{S}_j$ with the block circulant matrix $\textbf{C}_j$ in the BTTB matrix $\textbf{B}$, we obtain a Bolck-Toeplitz-Circulant-Block (BTCB) matrix $\widehat{\textbf{B}}$
\begin{flalign*}
\widehat{\textbf{B}}=\left(
\begin{array}{cccccccc}
\textbf{C}_{0}   &\textbf{C}_{1}   &\textbf{C}_{2}   &\cdots  &\textbf{C}_{M-2} &\textbf{C}_{M-1} &\textbf{C}_{M}\\
\textbf{C}_{1}   &\textbf{C}_{0}   &\textbf{C}_{1}   &\cdots  &\textbf{C}_{M-3} &\textbf{C}_{M-2} &\textbf{C}_{M}\\
\textbf{C}_{2}   &\textbf{C}_{1}   &\textbf{C}_{0}   &\textbf{C}_{1}   &\vdots  &\vdots  &\vdots\\
\vdots  &\vdots  &\ddots  &\ddots  &\ddots  &\vdots  &\vdots\\
\textbf{C}_{M-2} &\textbf{C}_{M-3} &\cdots  &\cdots  &\textbf{C}_{0}   &\textbf{C}_{1}   &\textbf{C}_{2}\\
\textbf{C}_{M-1} &\textbf{C}_{M-2} &\textbf{C}_{M-3} &\cdots  &\textbf{C}_{1}   &\textbf{C}_{0}   &\textbf{C}_{1}\\
\textbf{C}_{M}   &\textbf{C}_{M-1} &\textbf{C}_{M-2} &\textbf{C}_{M-3} &\cdots  &\textbf{C}_{1}   &\textbf{C}_{0}\\
\end{array}
\right).
\end{flalign*}

The BTCB matrix $\widehat{\textbf{B}}$ can be embedded into a $(2M+2)\times(2M+2)$ Bolck-Circulant-Circulant-Block (BCCB) matrix $\textbf{D}$ as follows
\begin{flalign*}
\textbf{D}=\left(
\begin{array}{cccccc}
\widehat{\textbf{B}}&\textbf{K}\\
\textbf{K}&\widehat{\textbf{B}}
\end{array}
\right),
\qquad K=\left(
\begin{array}{ccccccc}
0         &\textbf{C}_M         &\cdots      &\textbf{C}_{2}     &\textbf{C}_{1}\\
\textbf{C}_M       &0           &\cdots      &\textbf{C}_{3}     &\textbf{C}_{2}\\
\vdots    &\vdots      &\ddots      &\vdots    &\vdots\\
\textbf{C}_{2}     &\textbf{C}_{3}       &\cdots      &0         &\textbf{C}_{M}\\
\textbf{C}_{1}     &\textbf{C}_{2}       &\textbf{C}_{3}       &\cdots    &0\\
\end{array}
\right).
\end{flalign*}

The BCCB matrix $\textbf{D}$ has the following decomposition
\begin{flalign}\label{be1}
\textbf{D}=\left(F\otimes F\right)^{-1}diag(Fd)\left(F\otimes F\right),
\end{flalign}
where $d$ is the first column vector of $\textbf{D}$ and $F\otimes F$ is the two dimensional discrete Fourier transform matrix. Then, it is well known that the matrix-vector multiplication $Fw$ for $w\in\mathbb{R}^{(2M+2)^2}$ can be carried out in $O(M^2$log$M^2)$ operations via the fast Fourier transform (FFT). Equation (\ref{be1}) shows that $\textbf{D}w$ can be evaluated in $O(M^2$log$M^2)$ operations. Define $N=(M+1)^2$ and for $G=(g_0,g_1,\ldots,g_M)^T$, define $\widehat{G}=(g_0,0,g_1,0,\ldots,g_M,0)$. Then, we obtain that $\textbf{B}G$ can be evaluated in $O(N$log$N)$ operations for any $G\in\mathbb{R}^{N}$ by evaluating $\textbf{D}\widetilde{G}$ with FFT where $\widetilde{G}=(\widehat{G},0)$. The overall computational cost of the fast conjugate gradient method is $O(N$log$^2N)$, since the number of iterations is $O($log$N)$. What we need to focus on is that if one uses the Gaussian elimination method straightforwardly to this linear system, then it requires $O(N^3)$ complexity. In addition, since $N\times N$ BTTB matrix is determined by only $2N-1$ entries rather than $N^2$ entries, the fast solver will reduce memory requirement from $O(N^{2})$ to $O(N)$.

\section{Numerical experiments}
In this section, we present some numerical examples for the nonlocal Cahn-Hilliard equation in two dimension to test our theoretical analysis which contains energy stability and convergence rates of the proposed numerical schemes. We use the finite difference method for spatial discretization for all numerical examples. In all examples, we set the domain $\Omega=(-1,1)\times(-1,1)$. All the solvers are implemented using Matlab and all the numerical experiments are performed on a computer with 8-GB memory.

The Gaussian kernel $J_\delta$ will be given below \cite{du2018stabilized}:
\begin{equation*}
J_\delta(\textbf{x})=\frac{4}{\pi^{d/2}\delta^{d+2}}e^{-\frac{|\textbf{x}|^2}{\delta^2}},\quad \textbf{x}\in \mathbb{R}^{d},\quad \delta>0.
\end{equation*}

From \cite{du2018stabilized}, one can see that for any $v\in C^\infty(\Omega)$, $\textbf{x}\in\Omega$, $\mathcal{L}_\delta v(\textbf{x})\rightarrow-\Delta v(\textbf{x})$ as $\delta\rightarrow0$. It tells us that the nonlocal Cahn-Hilliard model equipped with above Gaussian kernel converges to the classical local Cahn-Hilliard model as $\delta\rightarrow0$.

We first give an example to test convergence rates of the proposed schemes (first order: SCHEM1 and second order:SCHEM2 in time) for the nonlocal Cahn-Hilliard equation in two dimension and check the efficiency of our fast procedure.

\textbf{Example 1}: Consider the nonlocal Cahn-Hilliard equation with $\epsilon^2=0.1$, $M=1$, $\delta=\epsilon$, $T=0.05$ and the following initial condition:
\begin{equation}\label{section5_e1}
\aligned
\phi_0(x,y)=0.5\sin(\pi x)\sin(\pi y)+0.1.
\endaligned
\end{equation}

To observe the temporal convergence rate, we first calculate the reference solution $\phi_{ref}$ with $h_{ref}=0.01$ and $\Delta t_{ref}=0.05\times2^{-14}$ since the exact solution is not known. Then, we use both direct solver (which means using the self-contained function $x=A\setminus b$ in Matlab R2015a) and fast solver (which means fast CG solver in Section 4) to obtain approximate solution with the time step sizes $\Delta t=0.05\times2^{-4}$, $0.05\times2^{-5}$, $0.05\times2^{-6}$, $0.05\times2^{-7}$, $0.05\times2^{-8}$, $0.05\times2^{-9}$ and $h=0.01$. Tables \ref{tab:tab1} and \ref{tab:tab2} show the discrete $L_2$ errors and the temporal convergence rates of numerical solutions and the CPU time costs for both direct and fast solvers. One can see that with the same mesh and time step sizes, both direct and fast solvers generate numerical solutions with almost the same errors and convergence rates. Furthermore, we again observe that fast solver saves more CPU time than direct solver to obtain same accuracy.
\begin{table}[h!b!p!]
\small
\renewcommand{\arraystretch}{1.1}
\centering
\caption{\small SCHEM1: the $L_2$ errors, temporal convergence rates and CPU time for direct solver and fast solver with $h_{ref}=0.01$, $\Delta t_{ref}=2^{-14}T$ for Example 1 with initial value $\phi_0(x,y)=0.5\sin(\pi x)\sin(\pi y)+0.1$.}\label{tab:tab1}
\begin{tabular}{ccccccc}
\hline
$\Delta t$&\multicolumn{3}{c}{Direct Solver}&\multicolumn{3}{c}{Fast Solver}\\
\cline{2-4}\cline{5-7}
&$L_2$ error&Rate&CPU Time(s)&$L_2$ error&Rate&CPU Time(s)\\
\hline
$2^{-4}T$ &2.5139e-3&-     &538  &2.5139e-3&-     &10.88\\
$2^{-5}T$ &1.5066e-3&0.7386&1035 &1.5066e-3&0.7386&16.95\\
$2^{-6}T$ &8.2386e-4&0.8708&2046 &8.2386e-4&0.8708&27.04\\
$2^{-7}T$ &4.2956e-4&0.9395&4045 &4.2956e-4&0.9395&47.48\\
$2^{-8}T$ &2.1803e-4&0.9783&7920 &2.1803e-4&0.9783&82.00\\
$2^{-9}T$ &1.0852e-4&1.0066&15758&1.0852e-4&1.0066&162\\
\hline
\end{tabular}
\end{table}
\begin{table}[h!b!p!]
\small
\renewcommand{\arraystretch}{1.1}
\centering
\caption{\small SCHEME2: the $L_2$ errors, temporal convergence rates and CPU time for direct solver and fast solver with $h_{ref}=0.01$, $\Delta t_{ref}=2^{-14}T$ for Example 1 with initial value $\phi_0(x,y)=0.5\sin(\pi x)\sin(\pi y)+0.1$.}\label{tab:tab2}
\begin{tabular}{ccccccc}
\hline
$\Delta t$&\multicolumn{3}{c}{Direct Solver}&\multicolumn{3}{c}{Fast Solver}\\
\cline{2-4}\cline{5-7}
&$L_2$ error&Rate&CPU Time(s)&$L_2$ error&Rate&CPU Time(s)\\
\hline
$2^{-4}T$ &9.2979e-4&-     &592   &9.2979e-4&-     &9.66\\
$2^{-5}T$ &2.5391e-4&1.8726&1091  &2.5391e-4&1.8726&15.65\\
$2^{-6}T$ &6.6781e-5&1.9268&2103  &6.6781e-5&1.9268&27.02\\
$2^{-7}T$ &1.7155e-5&1.9608&4147  &1.7155e-5&1.9608&47.98\\
$2^{-8}T$ &4.3486e-6&1.9800&7965  &4.3486e-6&1.9800&82.82\\
$2^{-9}T$ &1.0932e-6&1.9920&15834 &1.0932e-6&1.9920&140.6\\
\hline
\end{tabular}
\end{table}

To observe spatial convergence rates, we calculate the reference solution $\phi_{ref}$ with $h_{ref}=2^{-10}$ and $\Delta t_{ref}=5e-5$. Then, we use $h=2^{-3},\ 2^{-4},\ 2^{-5},\ 2^{-6},\ 2^{-7},\ 2^{-8}$, and $\Delta t=5e-5$ to obtain the approximate solution at $t=T$ for direct solver and fast solver. Tables \ref{tab:tab3} and \ref{tab:tab4} show the discrete $L_2$ errors, the spatial convergence rates of numerical solutions and CPU time costs for both direct and fast solvers. It can be observed that the spatial error are almost $O(h^2)$ for both SCHEM1 and SCHEM2. One can also obtain that the fast solver has almost the same error and convergence rate as direct solver. In addition, we again observe that fast solver saves more CPU time and memory requirement than direct solver to obtain similar accuracy. Thus, we can realize the numerical simulation under the finer meshes more quickly through the fast algorithm.

\begin{table}[h!b!p!]
\small
\renewcommand{\arraystretch}{1.1}
\centering
\caption{\small SCHEM1: the $L_2$ errors, spatial convergence rates and CPU time for direct solver and fast solver with $h_{ref}=2^{-10}$, $\Delta t_{ref}=5e^{-5}$ for Example 1 with initial value $\phi_0(x,y)=0.5\sin(\pi x)\sin(\pi y)+0.1$.}\label{tab:tab3}
\begin{tabular}{ccccccc}
\hline
$h$&\multicolumn{3}{c}{Direct Solver}&\multicolumn{3}{c}{Fast Solver}\\
\cline{2-4}\cline{5-7}
&$L_2$ error&Rate&CPU Time(s)&$L_2$ error&Rate&CPU Time(s)\\
\hline
$2^{-3}$ &6.7255e-3&-     &6.92         &6.7255e-3&-     &12.13\\
$2^{-4}$ &1.4786e-3&2.1854&87.86        &1.4786e-3&2.1854&19.42\\
$2^{-5}$ &3.4718e-4&2.0905&2689         &3.4718e-4&2.0905&72.54\\
$2^{-6}$ &8.3250e-5&2.0602&12436        &8.3250e-5&2.0602&280.8\\
$2^{-7}$ &N/A      &N/A &Out of Memory  &1.9526e-5&2.0921&1527\\
$2^{-8}$ &N/A      &N/A &Out of Memory  &3.8761e-6&2.3327&7979\\
\hline
\end{tabular}
\end{table}

\begin{table}[h!b!p!]
\small
\renewcommand{\arraystretch}{1.1}
\centering
\caption{\small SCHEM2: The $L_2$ errors, spatial convergence rates and CPU time for direct solver and fast solver with $h_{ref}=2^{-10}$, $\Delta t_{ref}=5e^{-5}$ for Example 1 with initial value $\phi_0(x,y)=0.5\sin(\pi x)\sin(\pi y)+0.1$.}\label{tab:tab4}
\begin{tabular}{ccccccc}
\hline
$h$&\multicolumn{3}{c}{Direct Solver}&\multicolumn{3}{c}{Fast Solver}\\
\cline{2-4}\cline{5-7}
&$L_2$ error&Rate&CPU Time(s)&$L_2$ error&Rate&CPU Time(s)\\
\hline
$2^{-3}$ &6.7282e-3&-     &6.76         &6.7282e-3&-     &11.88\\
$2^{-4}$ &1.4793e-3&2.1853&92.36        &1.4793e-3&2.1853&16.28\\
$2^{-5}$ &3.4733e-4&2.0905&2866         &3.4733e-4&2.0905&68.86\\
$2^{-6}$ &8.3287e-5&2.0601&12738        &8.3287e-5&2.0601&275.4\\
$2^{-7}$ &N/A      &N/A &Out of Memory  &1.9535e-5&2.0920&1531\\
$2^{-8}$ &N/A      &N/A &Out of Memory  &3.8778e-6&2.3328&7233\\
\hline
\end{tabular}
\end{table}

In the following examples, we study the phase separation behavior using the second order scheme SCHEM2 and obtain the numerical solution by fast solver.

\textbf{Example 2}: In the following, we take $\epsilon=0.02$, $M=1$, $\delta=0.02$. The initial condition is chosen as
\begin{equation}\label{section5_e2}
\aligned
\phi_0(x,y,0)=\sum\limits_{i=1}^2-\tanh\left(\frac{\sqrt{(x-x_i)^2+(y-y_i)^2}-R_0}{\sqrt{2}\epsilon}\right)+1.
\endaligned
\end{equation}
with the radius $R_0=0.36$, $(x_1,y_1)=(0.4,0)$ and $(x_2,y_2)=(-0.4,0)$. Initially, two bubbles, centered at $(0.4,0)$ and $(-0.4,0)$, respectively, are osculating or "kissing". In the simulation, we choose the mesh size $h=1/100$ and the time step $\Delta t=10^{-3}$. The process coalescence of two bubbles is demonstrated in Figure \ref{fig:fig1}. Snapshots of the phase variable $\phi$ are taken at $t=0$, $0.05$, $0.1$, $1$, $5$, $10$ in Figure \ref{fig:fig1}. In Figure \ref{fig:fig1}, we show the evolutions of the phase field variable $\phi$ at various time by using the time step $\Delta t=1e-3$. We observe the coarsening effect that the small circle is absorbed into the big circle, and the total absorption happens at around $t=10$ which is consistent with classical Cahn-Hilliard model.
\begin{figure}[htp]
\centering
\subfigure[t=0]{
\includegraphics[width=3cm,height=3cm]{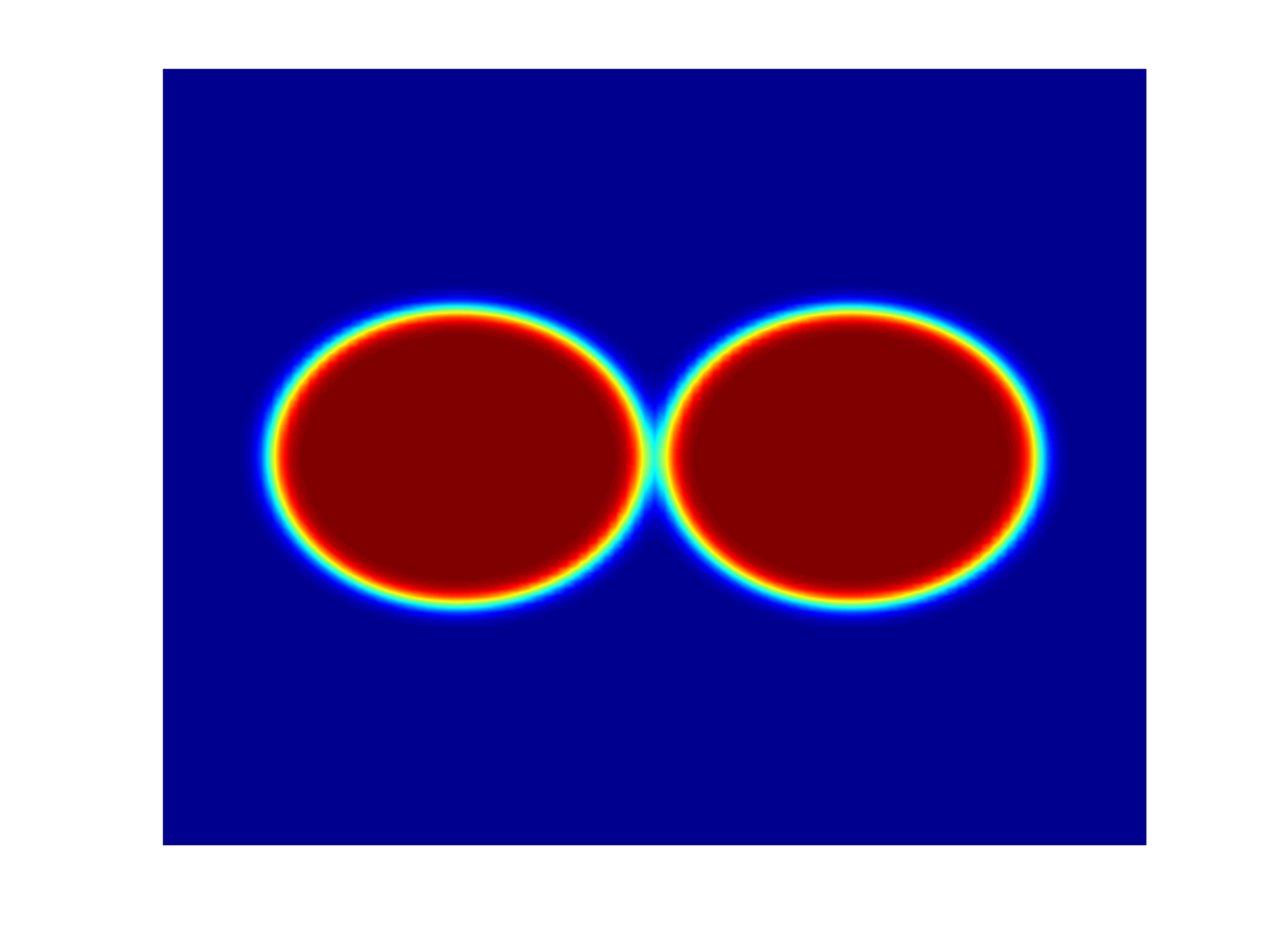}
}
\subfigure[t=0.05]
{
\includegraphics[width=3cm,height=3cm]{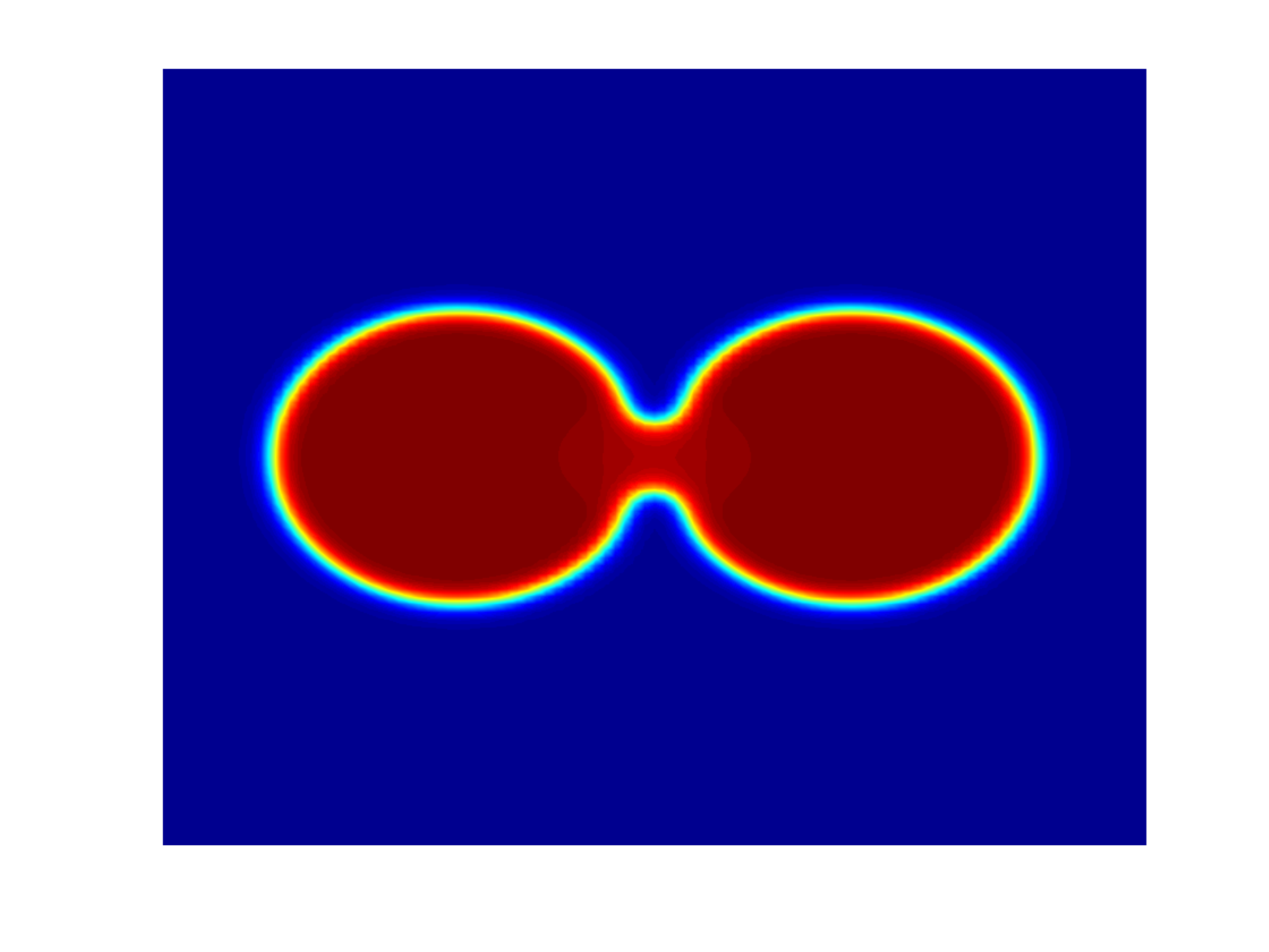}
}
\subfigure[t=0.1]
{
\includegraphics[width=3cm,height=3cm]{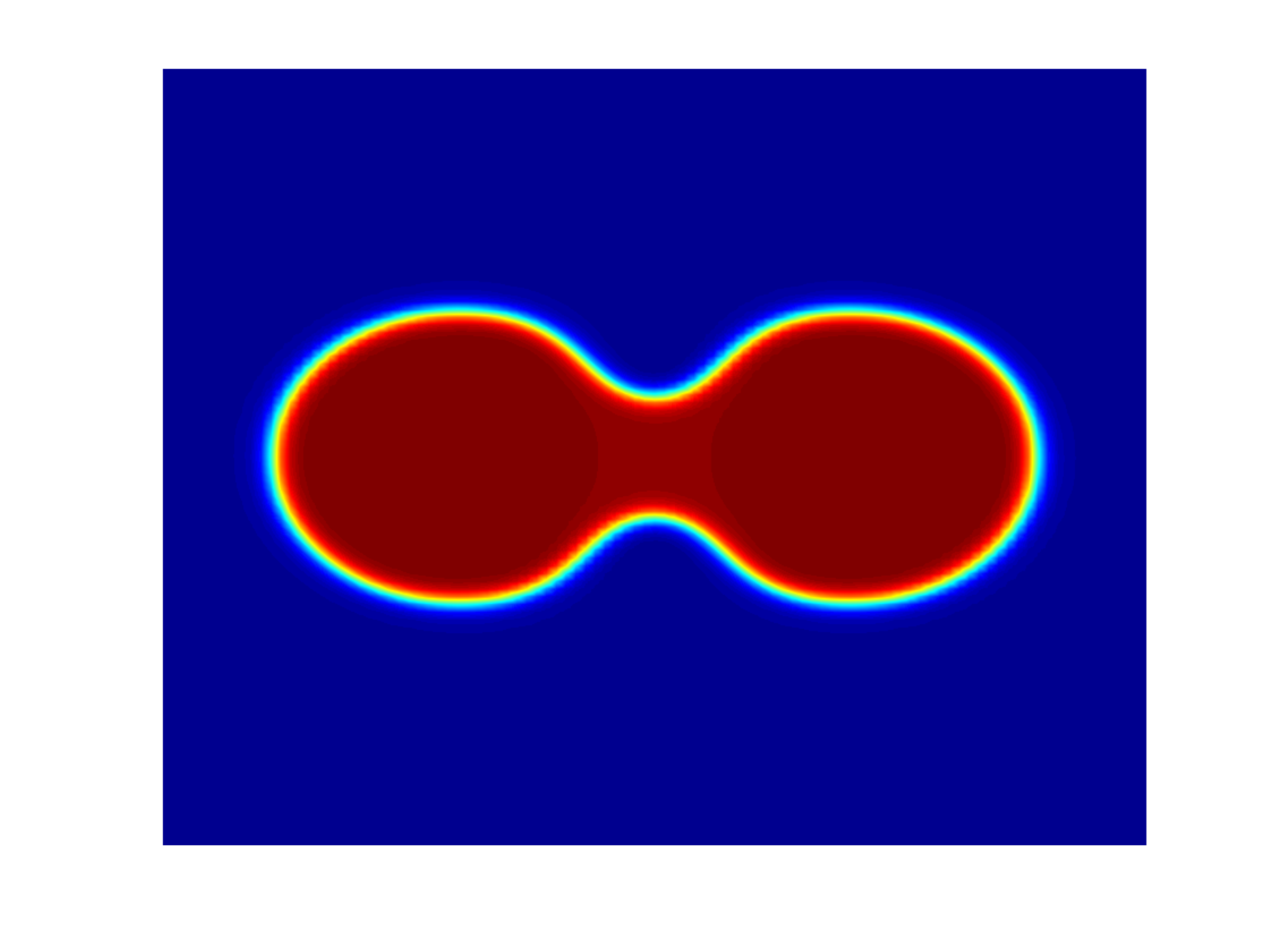}
}
\quad
\subfigure[t=1]{
\includegraphics[width=3cm,height=3cm]{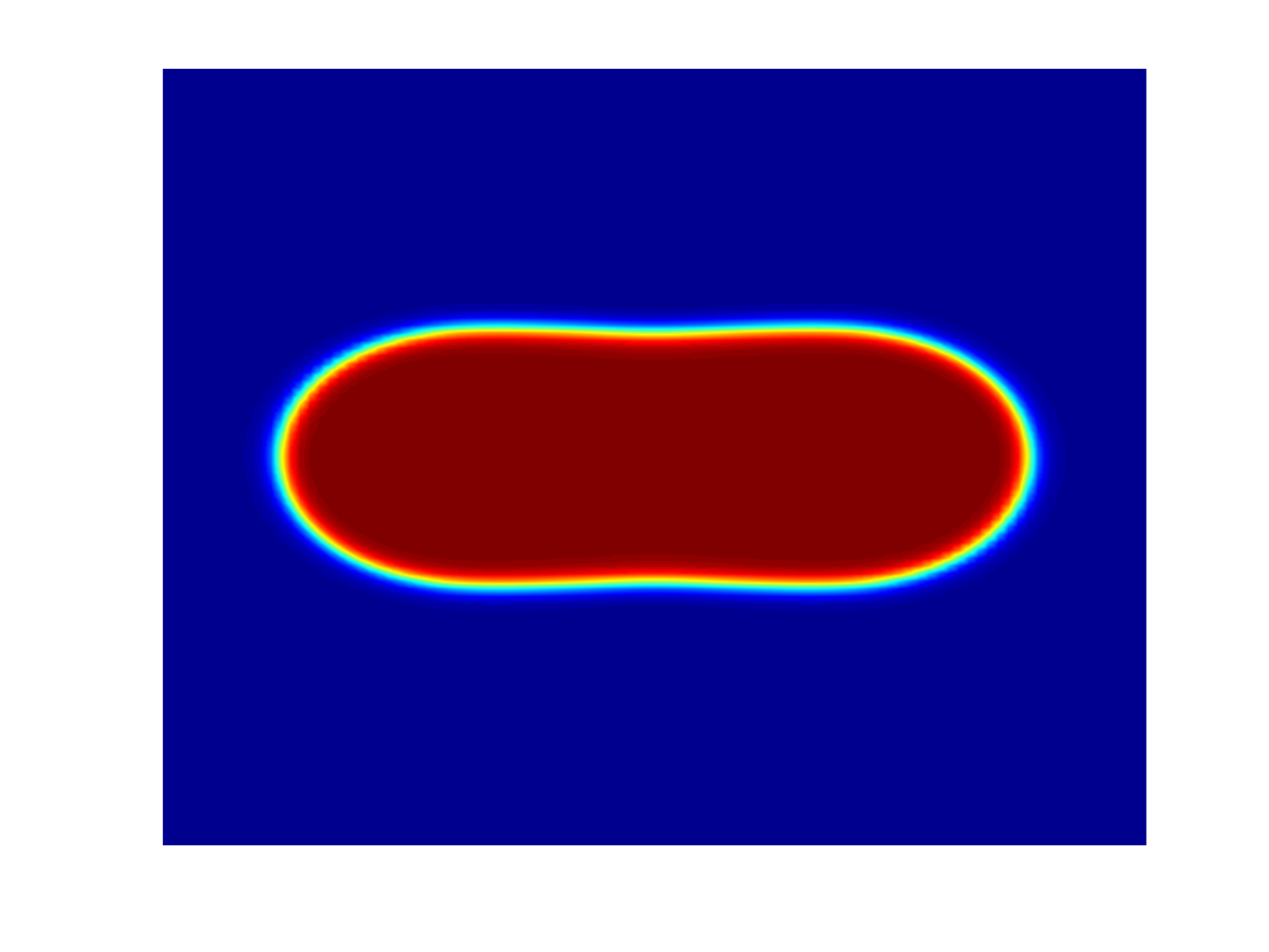}
}
\subfigure[t=5]
{
\includegraphics[width=3cm,height=3cm]{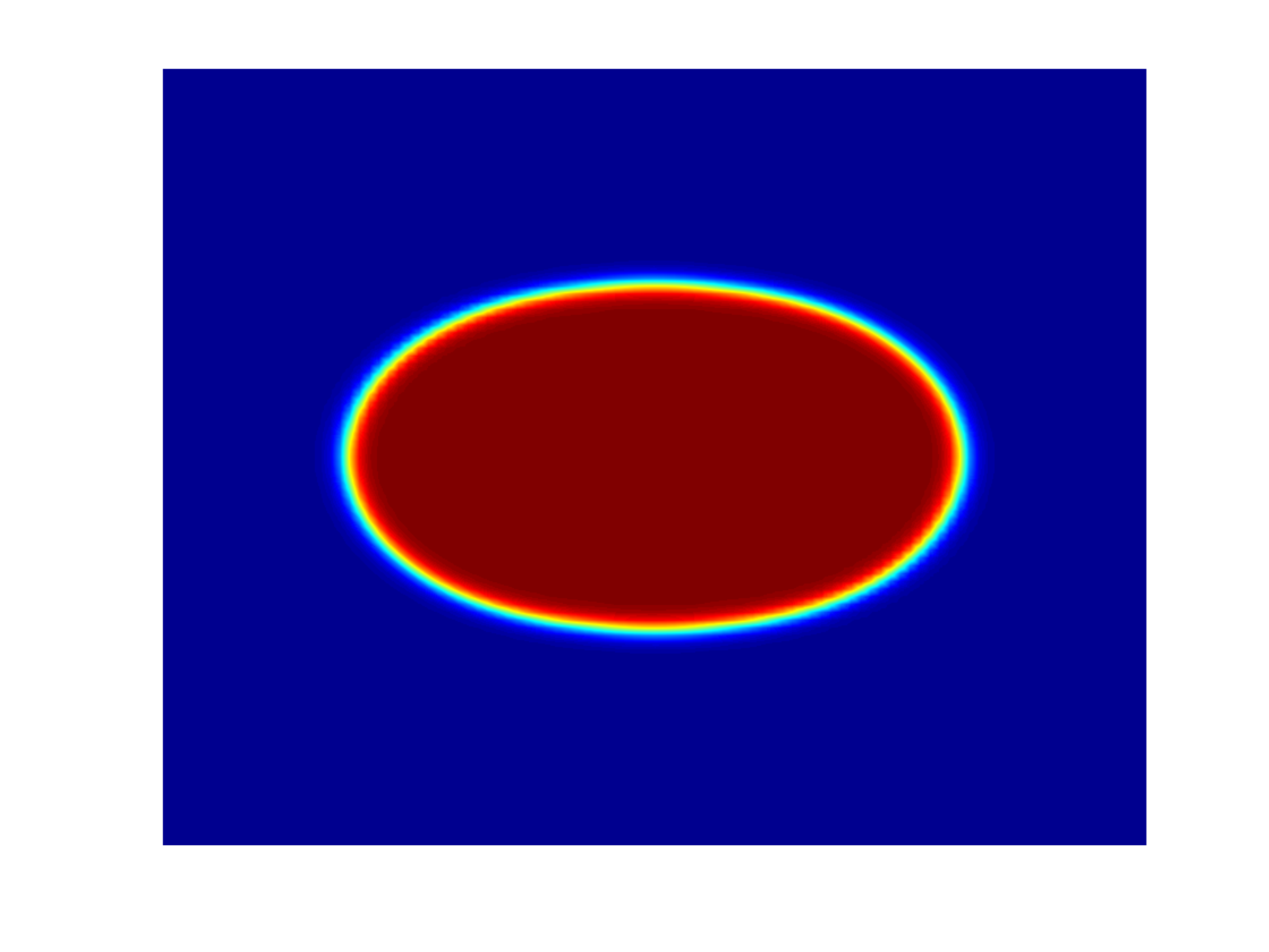}
}
\subfigure[t=10]
{
\includegraphics[width=3cm,height=3cm]{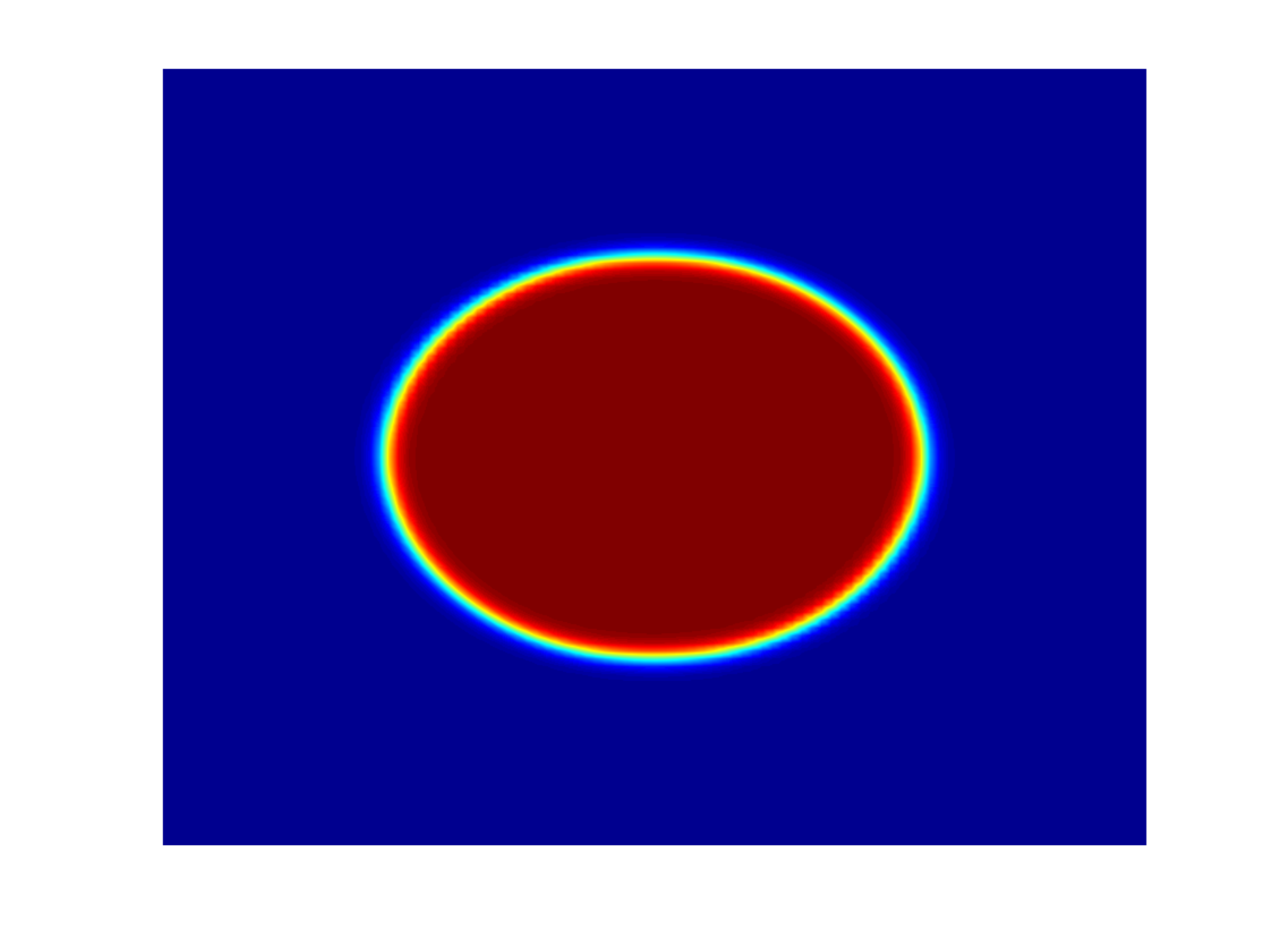}
}
\caption{Snapshots of the phase variable $\phi$ are taken at t=0, 0.05, 0.1, 1, 5, 10 for example 2 with the initial condition (\ref{section5_e2}).}\label{fig:fig1}
\end{figure}

\textbf{Example 3}: We give the following example for nonlocal Cahn-Hilliard equation with $\epsilon=0.02$, $M=1$. The initial condition is
\begin{equation}\label{section5_e3}
\aligned
&\phi_0(x,y)=0.1\times rand(x,y),
\endaligned
\end{equation}
where the $rand(x,y)$ is the random number in $[-1,1]$ with zero mean.

In Figure \ref{fig:fig2}, we perform the simulations by using the time step $\Delta t=1e-3$. The figure shows the dynamical behaviors of the phase separation for the random initial value (\ref{section5_e3}).  In Figure \ref{fig:fig2}, the snapshots of coarsening dynamics are taken at $t=0$, $0.1$, $0.5$, $1$, $2$, $10$ and the final steady shape forms several big drops.
Next, we fix $\delta=0.05$ with $h=0.01$ and decrease $\epsilon$ from $0.1$ to $0.02$. The energies are plotted in Figure \ref{fig:fig3}. One can find that the energy decay rates comply with the $-1/3$ power law quite well for all cases which is consistent to the classical local Cahn-Hilliard model.

\begin{figure}[htp]
\centering
\subfigure[t=0]{
\includegraphics[width=3cm,height=3cm]{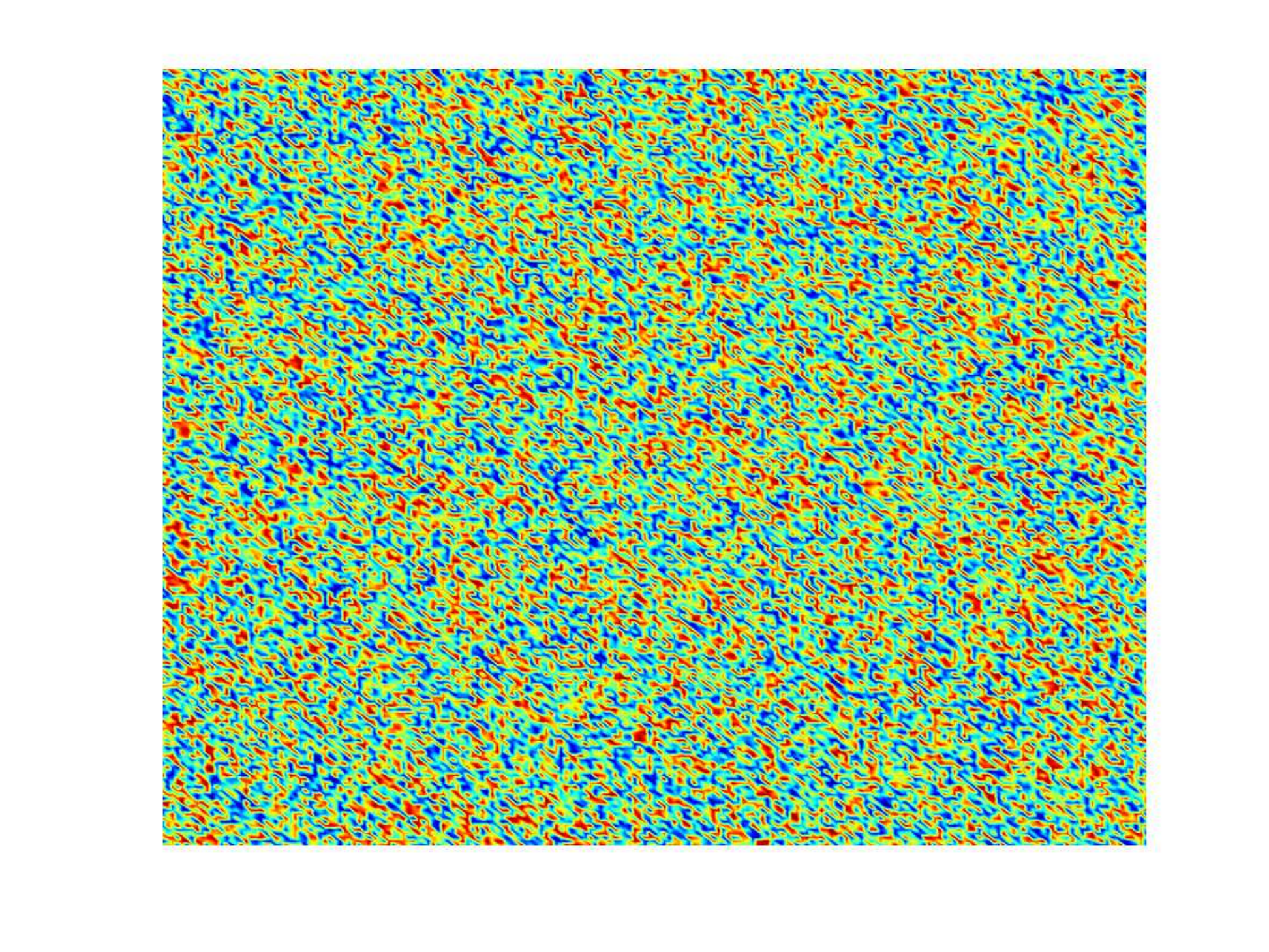}
}
\subfigure[t=0.1]
{
\includegraphics[width=3cm,height=3cm]{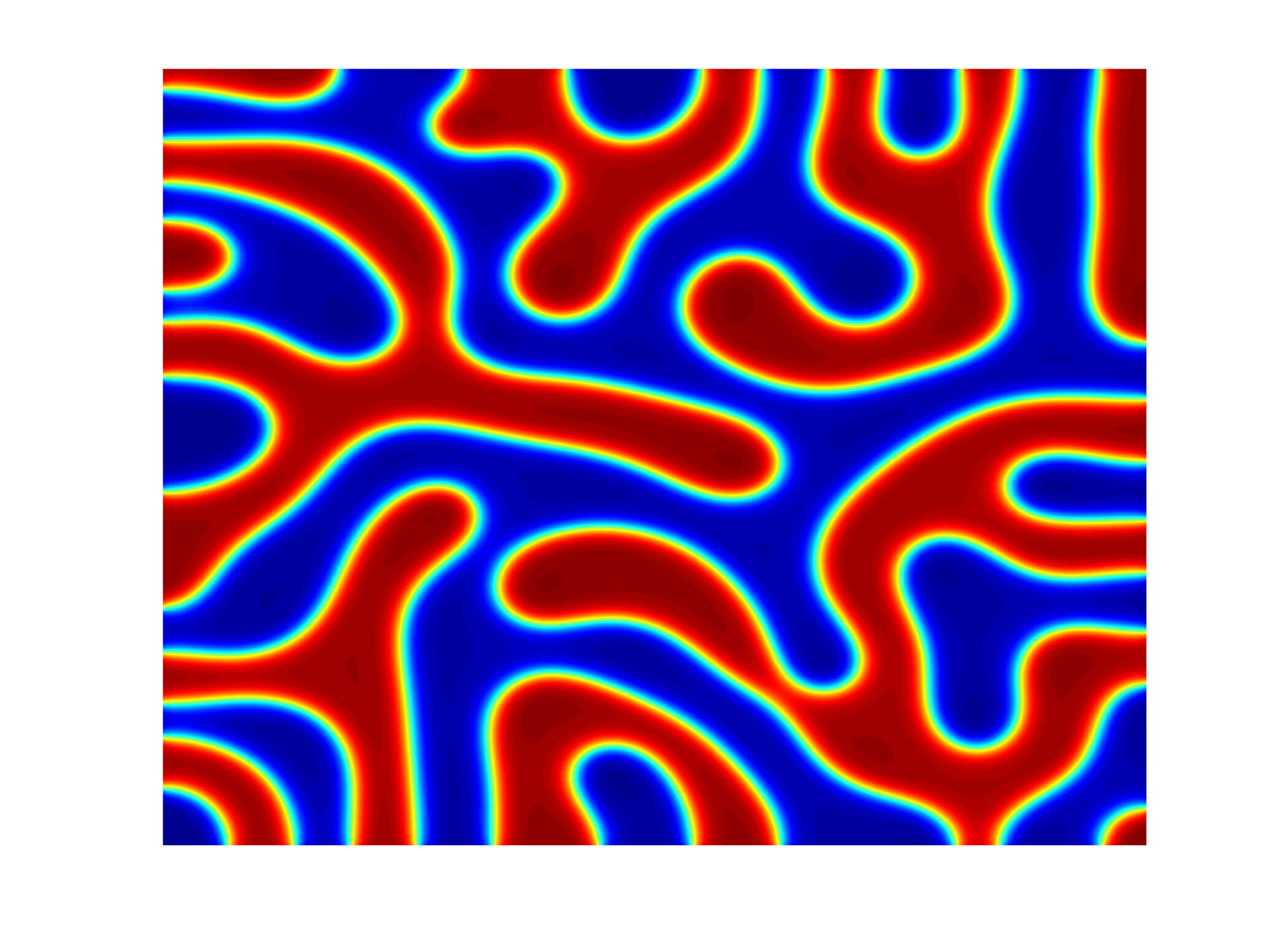}
}
\subfigure[t=0.5]
{
\includegraphics[width=3cm,height=3cm]{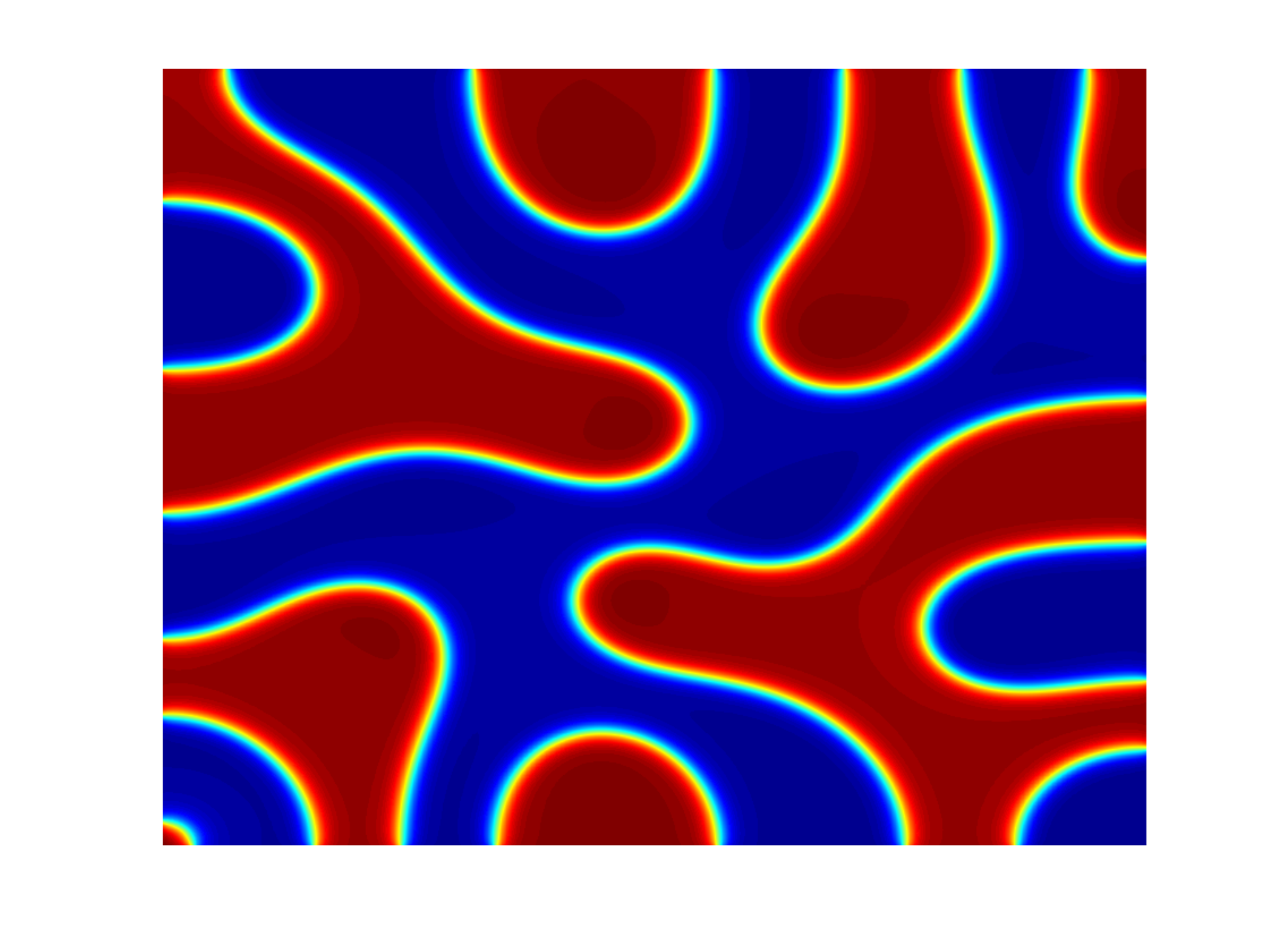}
}
\quad
\subfigure[t=1]{
\includegraphics[width=3cm,height=3cm]{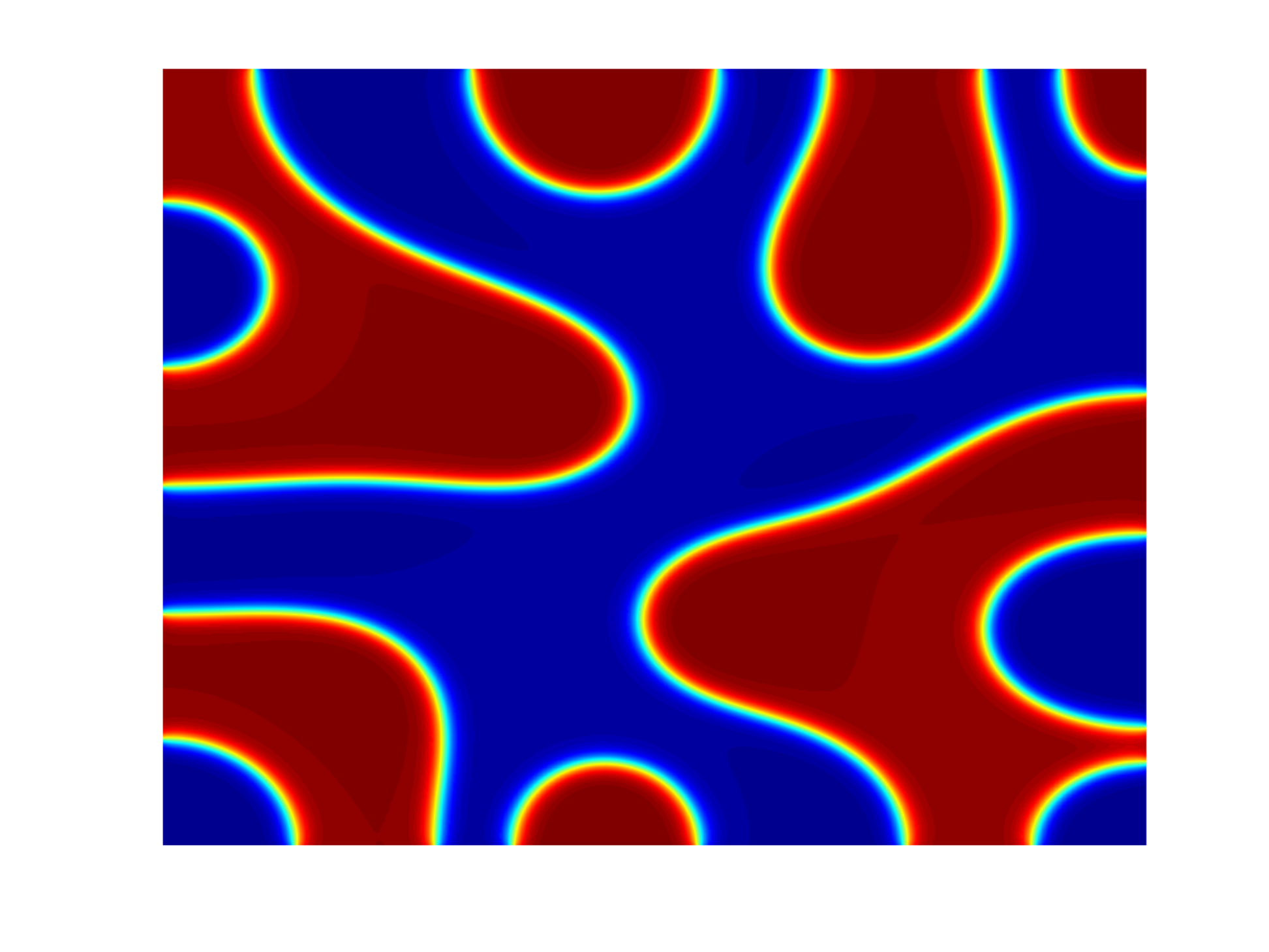}
}
\subfigure[t=2]
{
\includegraphics[width=3cm,height=3cm]{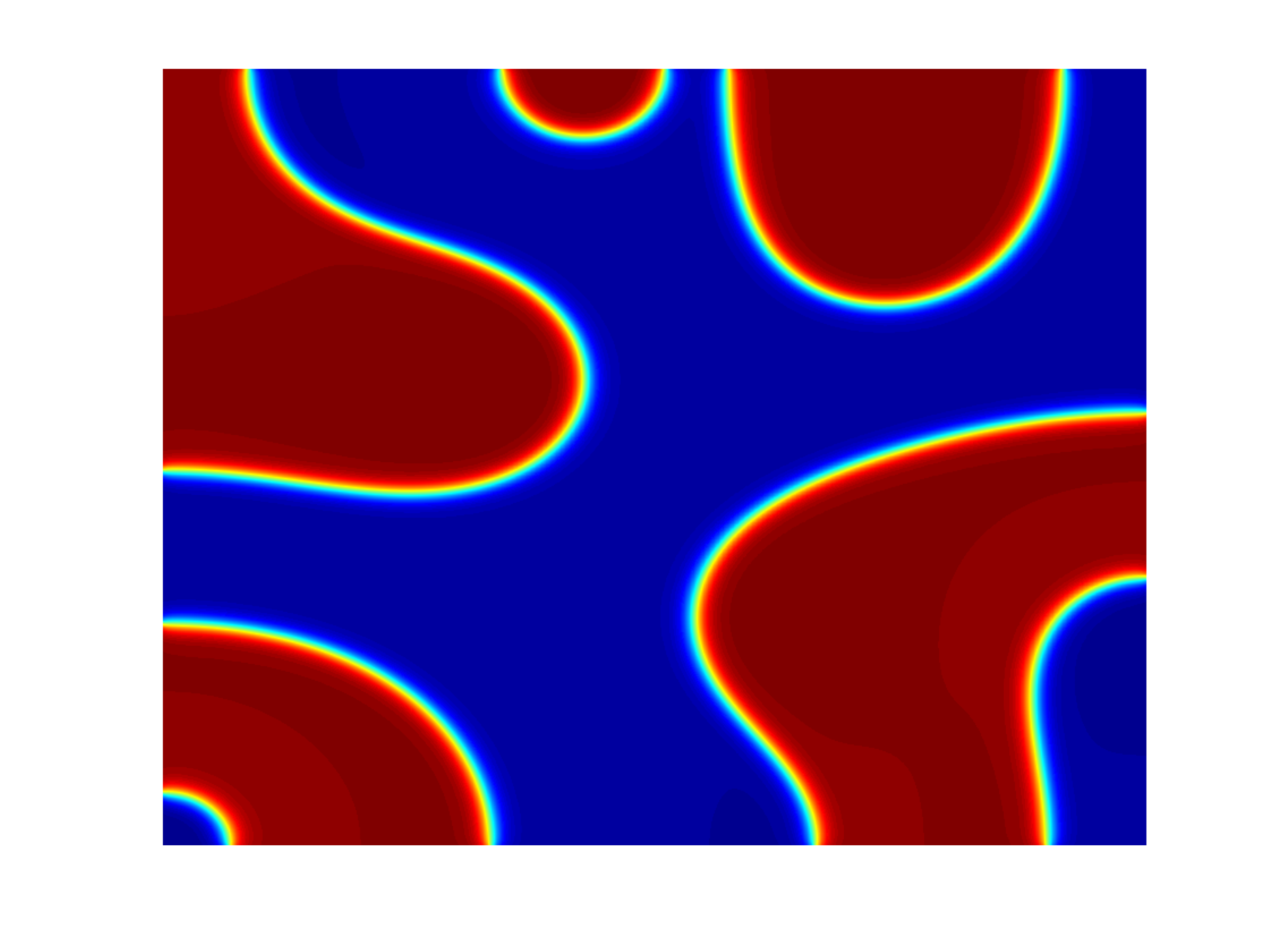}
}
\subfigure[t=10]
{
\includegraphics[width=3cm,height=3cm]{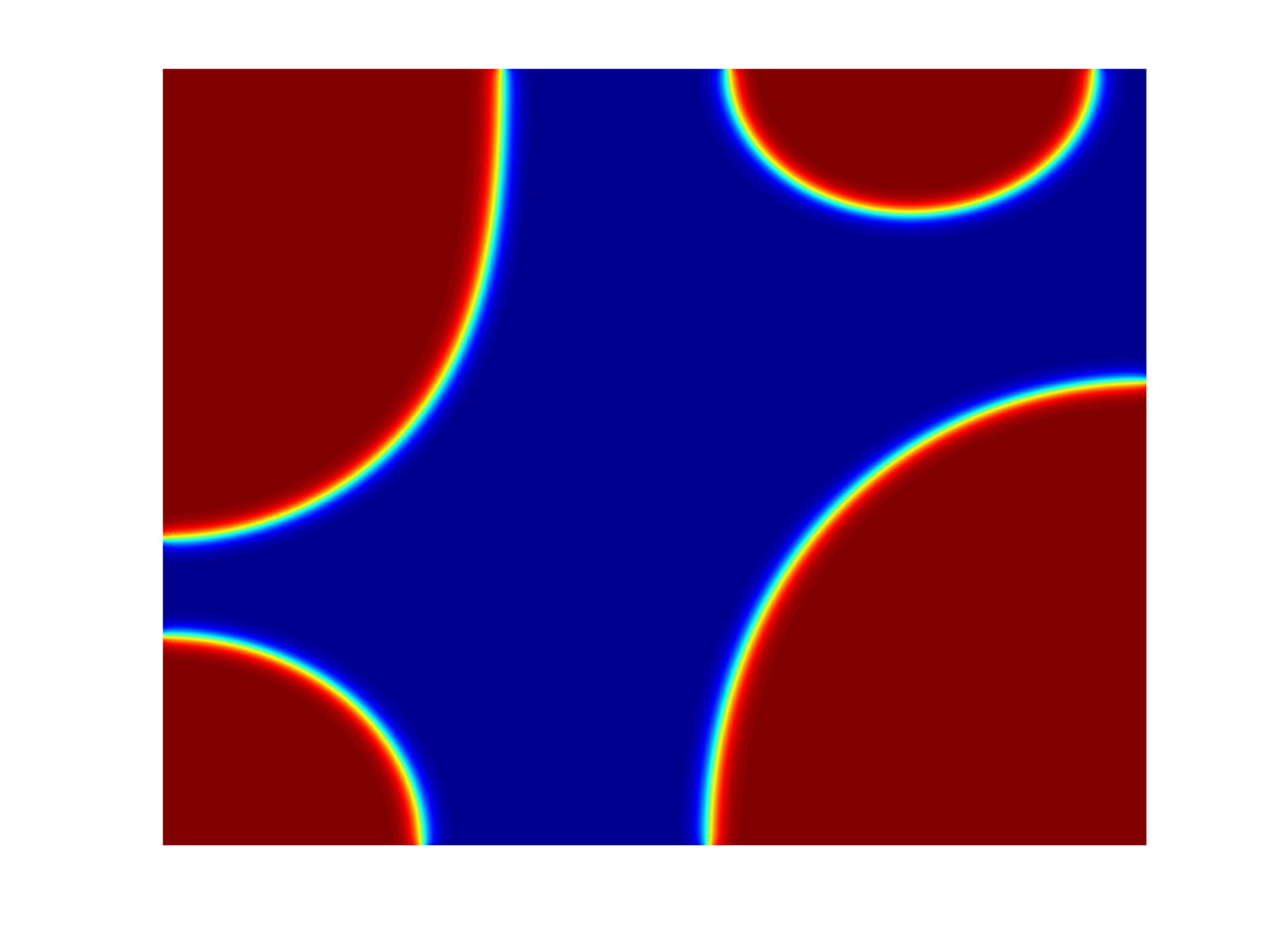}
}
\caption{Snapshots of the phase variable $\phi$ are taken at t=0, 0.1, 0.5, 1, 2, 10 for example 3 with the initial condition (\ref{section5_e3}).}\label{fig:fig2}
\end{figure}

\begin{figure}[htp]
\centering
\includegraphics[width=10cm,height=7cm]{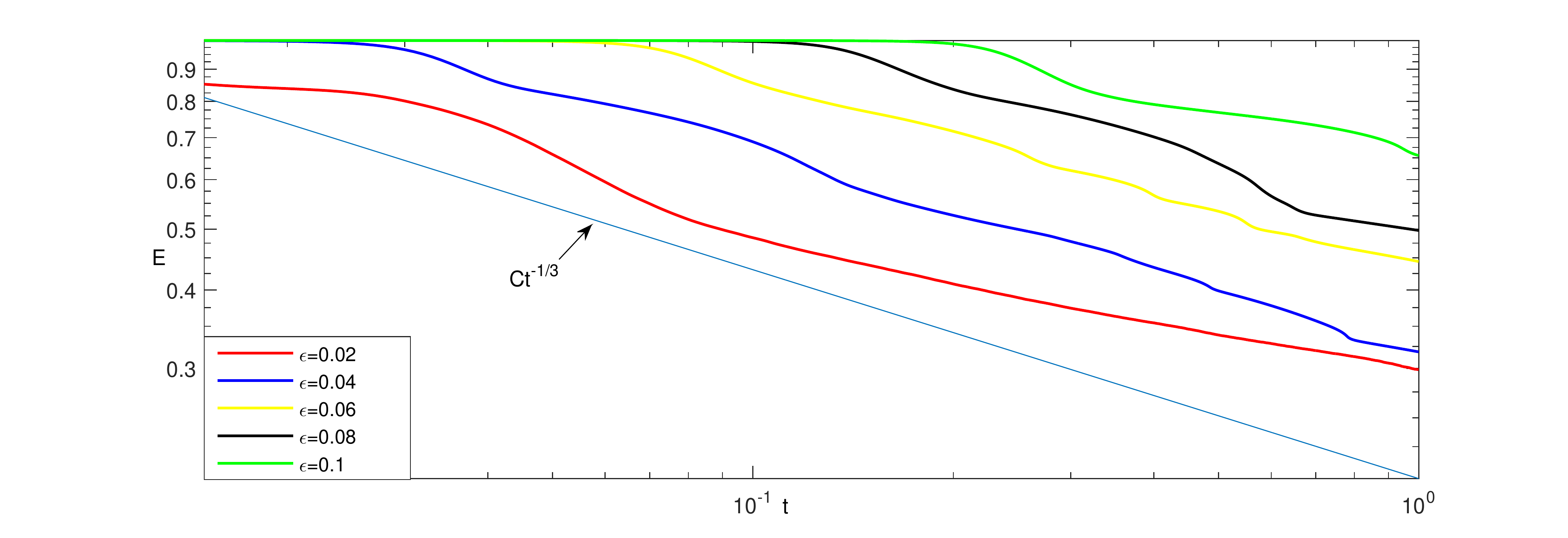}
\caption{Energy evolution of SCHEM2 with $\delta=0.05$, $M=1$ and $\epsilon=0.02$, $0.04$, $0.06$, $0.08$ and $0.1$ for example 3.}\label{fig:fig3}
\end{figure}
\section{Conclusion}
In this paper, we develop accurate and efficient linear algorithms for the general nonlocal Cahn-Hilliard equation with general nonlinear potential and prove the unconditional energy stability for its semi-discrete schemes carefully and rigorously. we construct and analyze linear, first and second order (in time) numerical scalar auxiliary variable approaches to construct unconditionally energy stable schemes. In addition, considering the huge computational work and memory requirement in solving the linear system, we analyse the structure of the stiffness matrix and seek some effective fast solution method to reduce the computational work and memory requirement. This fast solution technique is based on a fast Fourier transform and depends on the special structure of coefficient matrices. In the future work, error estimates for the fully discrete schemes will be investigated.

\bibliographystyle{siamplain}
\bibliography{SAV_nonlocal-cahn-hilliard}

\end{document}